\documentclass[reqno, 12pt]{amsart}

\usepackage{amsmath, amssymb, amsthm, amsfonts, txfonts, pxfonts}
\usepackage[english]{babel}
\usepackage{fullpage}
\usepackage{enumerate}

\usepackage[centertags]{amsmath}
\usepackage{indentfirst}
\usepackage{fancyhdr}
\usepackage[dvips]{graphicx}
\usepackage{psfrag}
\numberwithin{equation}{section}
\usepackage[latin1]{inputenc}

\usepackage[all]{xypic}

\title{Deformation of tensor product (co)algebras \\via non-(co)normal twists}
\author{Lucio S. Cirio, Chiara Pagani}
\address[]{\textit{Lucio Simone Cirio} \newline \indent 
Max Planck Institute for Mathematics, 
Vivatsgasse 7, 53111 Bonn,  Germany 
\newline \indent \textit{Present address:}University of Luxembourg, Campus Kirchberg,
Mathematics Research Unit,\newline \indent
6, rue Richard Coudenhove-Kalergi
L-1359 Luxembourg
Grand-Duchy of Luxembourg 
}
\email{lucio.cirio@uni.lu }

\address[]{\textit{Chiara Pagani} \newline \indent   University of Luxembourg, Campus Kirchberg,
Mathematics Research Unit,\newline \indent
6, rue Richard Coudenhove-Kalergi
L-1359 Luxembourg
Grand-Duchy of Luxembourg
}
\email{chiara.pagani@uni.lu}

\newtheorem{thm}{Theorem}[section]
\newtheorem{lem}[thm]{Lemma}
\newtheorem{prop}[thm]{Proposition}
\newtheorem{defi}[thm]{Definition}

\newtheorem{cor}[thm]{Corollary}

\theoremstyle{remark}
\newtheorem{ex}[thm]{Example}

\newtheorem{rem}[thm]{Remark}

\newcommand{\kk}{{{k}}}

\newcommand{\ta}{\psi}

\newcommand{\nt}{\chi_{{}_\Psi}}

\newcommand{\rac}{\triangleleft}
\newcommand{\co}{\Delta}

\newcommand{\ot}{\otimes}
\newcommand{\tn}{\otimes}
\newcommand{\be}{\begin{equation}}
\newcommand{\ee}{\end{equation}}
\newcommand{\ra}{\rightarrow}

\newcommand{\id}{\mathrm{id}}
\newcommand{\pp}[2]{\langle #1 \, , \, #2 \rangle}
\newcommand{\tw}{\mathsf{Tw}}
\newcommand{\twa}{\mathsf{tw}}

\def \M {^{D}\mathit{M}^{C}}
\def \m {_{B}\mathit{M}_{A}}
\def \tco {\varepsilon_{\Psi}}

\newcommand{\uti}{co-commutative up to isomorphism}
\newcommand{\uno}[1]{{#1}_{\scriptscriptstyle{(1)}}}
\newcommand{\due}[1]{{#1}_{\scriptscriptstyle{(2)}}}
\newcommand{\tre}[1]{{#1}_{\scriptscriptstyle{(3)}}}
\newcommand{\qu}[1]{{#1}_{\scriptscriptstyle{(4)}}}
\newcommand{\suno}[1]{{#1}^{\scriptscriptstyle{(1)}}}
\newcommand{\suze}[1]{{#1}^{\scriptscriptstyle{(0)}}}

\newcommand{\taa}[1]{{#1}^{[\ta]}}
\def \tu {\eta_{\ta}}
\newcommand{\utia}{commutative up to isomorphism}

\begin{document}

\begin{abstract}
We study new coalgebra structures on the tensor product of two coalgebras $C$ and $D$ by twisting the tensor product coalgebra via a twist map $\Psi: C \ot D \ra D\ot C$. 
We deal with the general case in which the counit of the tensor product coalgebra is deformed as well. Some classes of such deformations are analyzed and a notion of equivalence of twists is discussed. 
  We also present the dual deformation of tensor product algebras and provide examples.
\end{abstract}

\maketitle
\noindent

\section{Introduction}

Twisted tensor product  algebras $A \ot_\Psi B$  arise as deformations of the tensor product algebra $(A \ot B,~ (m_A \ot m_B) \circ (\id_A \ot \tau \ot \id_B))$ of two associative algebras $A$ and $B$ via linear maps $\Psi: B \ot A \ra A \ot B$, with new multiplication $m_\Psi=(m_A \ot m_B) \circ (\id_A \ot \Psi \ot \id_B)$.  A compatibility condition (cf. \eqref{Oeq}) between $\Psi$ and the multiplication maps  $m_A,~m_B$ provides a necessary and sufficient condition for the associativity of the twisted multiplication $m_\Psi$. 
\\
Dually a twisted tensor coalgebra $C \ot_\Psi D$ of two coassociative coalgebras $C,~D$ is constructed on the vector space $C
\ot D$ in terms of a twisted coproduct $\co_\Psi= (\id_C \ot \Psi \ot \id_D) \circ (\co_C \ot \co_D)$ out of a linear map $\Psi: C \ot D \ra D \ot C$. For twists $\Psi$ satisfying a suitable condition (cf. \eqref{COdiagr}), $\co_\Psi$ is coassociative. 

Deformations of tensor products in this setting were first studied in \cite{csv} in the case of algebras. Nevertheless the idea of deforming the tensor product of two objects to obtain a new one is much older than that. Notable examples are  the classical crossed products (see e.g. \cite{mon}), the quantum double of Drinfel'd \cite{drinfeld} or the  bicrossproduct and double cross product  Hopf algebras of Majid \cite{majid}. In these constructions one deals with a much richer structure on and interplay between the two given objects which is then reflected on the resulting one. Disregarding this extra data, some of these concepts can be retrieved as particular instances of the  theory of twisted tensor  products.\\ 
 Twisted tensor products can be studied in the framework of (co)algebra factorisations of a (co)algebra in two sub(co)algebras   
or in the opposite perspective
of building a new object out of two given ones.  Deformations of tensor products have been studied both in pure algebra and in connection to other branches of mathematics, notably non commutative geometry, and to physics. Some relevant spaces in noncommutative geometry and physics can be indeed recovered as particular twisted tensor products (see e.g. \cite{jlpo} and the references therein).

Historically, twisted products of (co)algebras have been studied for twisting maps $\Psi$ which satisfy a (co)unital condition (cf. Def. \ref{def-coun}, Def. \ref{def-un}) which ensures that the undeformed tensor (co)unit is still compatible with the deformed (co)product. In the present paper we address to the more general case of twists which are not necessarily (co)normal. For different reasons we were most interested in the deformation of tensor product of coalgebras and so we present our results first in that case.  Once we drop the requirement that the twist map $\Psi$ is conormal, the new coalgebra 
$C \ot_\Psi D$ might or might not admit a counit. In Sect. \ref{sec-nct} we study twisted coalgebras associated to $Z$-conormal twists. This notion is a  generalization of  the former one and guarantees the existence of a compatible counit $\varepsilon_\Psi$
for the twisted coproduct $\co_\Psi$.   These deformations still enjoy a universal property among  factorized coalgebras (Thm \ref{thm-univ-nct}),   in analogy to the  case of conormal twists treated in \cite{cae}. 
We also analyze  a class of  twisted coalgebras $C \ot_{\Psi_\phi} D$
which are generated from twists which are intrinsically non conormal. These twists correspond to particular morphisms in the category $^D M^C$ of left $D$ - right $C$-comodules and are in one-to-one correspondence with functionals $\phi$ on $C \ot D$. The existence of a counit $\varepsilon_{\Psi_\phi}$ in this case can be expressed as a condition on $\phi$. Although this class of twists consist of non-conormal ones,  the resulting twisted coalgebras $(C \ot_{\Psi_\phi} D)$ eventually turned out to be all  isomorphic to the untwisted one $C \ot D$ (albeit non trivially). Nevertheless, this class can be used to built new interesting deformations (see \S \ref{se-nt}).
We also briefly discuss a notion of equivalence of twists, but a cohomological interpretation of our deformations along the lines of \cite{br2001, csv} is postponed. We conclude the paper by presenting the results in the dual case of the tensor product of two algebras. 
\\

The structure of the paper is as follows. In Sect. \ref{se-review} we review some known results regarding twisted tensor product
coalgebras relevant to our study. In particular we recall the factorization and universal properties for the twisted tensor
coalgebra $C \ot_\Psi D$ of two coalgebras by a conormal twist $\Psi$. In Sect. \ref{sec-nct} we introduce the notion of $Z$-conormal twists and generalize in this framework the previously mentioned results.   Sect. \ref{se-functional} describes the particular class of coassociative twisted tensor coalgebras $C \ot_\Psi D$ which are generated from twists associated to
  functionals on $C \ot D$. Further, in Sect. \ref{se-nt} we discuss of their use for the construction of new twisted tensor coalgebras out of old ones.  In Sect. \ref{se-equiv} we introduce a possible notion of equivalence of twists. Finally in 
Sect. \ref{se-dual} we present the dual results for the deformation via twists of the tensor product algebra $A \ot B$ of two algebras.  

\section{Twisting tensor coalgebra co-structures}\label{se-review}

\textit{Notation.} Coalgebras are over a commutative field $\kk$. Their coproduct and counit are as usual denoted with $\co, \varepsilon$ respectively. When it is necessary to avoid confusion, we specify to which coalgebra $C$ they refer by writing $\co_C$ and $\varepsilon_C$. The unadorned tensor product $\ot$ stands for tensor product over $\kk$, and $\tau$ is the flip map. We denote $C^{cop}$ the coalgebra $C$ endowed with the opposite coproduct $\co^{cop}:=\tau\circ\co_C$.  Given a tensor product of coalgebras $C\tn D$ we write $\co_{\tn},\varepsilon_{\tn}$ for the tensor product co-structures $\co_{\tn}:= (\id_C\tn\tau\tn\id_D)(\co_C\tn\co_D)$ and $\varepsilon_{\tn}:=\varepsilon_C\tn\varepsilon_D$. In case of Hopf algebras, the antipode is denoted with $S$. Summation over repeated indices is understood and we make use of Sweedler and Sweedler-like notations for coproduct and coactions. \\ \medskip

Given $C$ and $D$ coassociative coalgebras, we are interested in defining a twisted coproduct on $C\ot D$.  Throughout the paper we call \emph{twist map} a bilinear map $\Psi:C\otimes D\ra D\otimes C \, , ~ c\tn d\mapsto d^{[\Psi]}\tn c^{[\Psi]}$.
Given a twist $\Psi$ we also use the notation  $\Psi':= \Psi \tau$.\\ 
The starting point of our investigation are the following results taken from \cite{cae, csv}. 
We follow the presentation of \cite[\S 3]{cae}. 

\begin{thm}
Let $\Psi$ be a twist map. The map $\co_{\Psi}:C\ot D\ra C\ot D\ot C\ot D$ given by
\begin{equation}
\label{copsi}
\co_{\Psi}:= (\id_C\tn \Psi \tn \id_D) \circ (\co_C\tn\co_D)
\end{equation}
defines a coassociative coproduct on $C\tn D$ if and only if the following diagram is commutative:
\begin{equation}
\label{COdiagr}
\xymatrix @C=4pc{
C\tn D \tn D \ar[r]^-{\Psi\tn \id_D} & D\tn C\tn D \ar[r]^-{\id_D\tn \co_C\tn \id_D} & D\tn C\tn C\tn D \ar[d]^-{\id_D\tn\id_C\tn\Psi} \\
C\tn D \ar[u]^-{\id_C\tn \co_D}\ar[d]_-{\co_C\tn \id_D} & & D\tn C\tn D \tn C \\
C\tn C\tn D \ar[r]^-{\id_C\tn \Psi} & C\tn D\tn C \ar[r]^-{\id_C\tn \co_D \tn \id_C} & C\tn D\tn D\tn C \ar[u]_-{\Psi\tn\id_D\tn\id_C}
}  
\end{equation}
that is
\begin{multline}
\label{COeq}
(\id_D\tn\id_C\tn\Psi)(\id_D\tn\co_C\tn \id_D)(\Psi\tn \id_D)(\id_C\tn\co_D) = \\
(\Psi\tn\id_D\tn\id_C)(\id_C\tn\co_D\tn \id_C) (\id_C\tn\Psi)(\co_C\tn \id_D)
\end{multline}
\end{thm}

We denote with $C \ot_\Psi D$ the vector space $C \ot D$ equipped with the twisted coproduct $\co_\Psi$. The problem of defining a counit on $C\tn D$ compatible with the twisted coproduct $\co_{\Psi}$ is usually solved by asking some further condition on $\Psi$. A possible choice is to restrict to twists which are compatible with the tensor counit $\varepsilon_{\tn}=\varepsilon_C\tn\varepsilon_D$.

\begin{defi}\label{def-coun}
A twist map $\Psi$ is said to be right conormal if
\begin{equation}
\label{rcon}
(\varepsilon_D\tn\id_C)\Psi = \id_C\tn\varepsilon_D \,.
\end{equation}
Similarly, it is said to be left conormal if
\begin{equation}
\label{lcon}
(\id_D\tn\varepsilon_C)\Psi = \varepsilon_C\tn \id_D \, .
\end{equation}
It is said to be conormal when it is both right and left conormal.
\end{defi}

\begin{lem} 
Let $\Psi$ be a twist map. The tensor counit $\varepsilon_\ot=\varepsilon_C\tn\varepsilon_D$ is compatible with the twisted coproduct $\co_{\Psi}$, i.e. $(\id \ot \varepsilon_\ot) \co_{\Psi}=\id_{C\tn D}= (\varepsilon_\ot \ot \id) \co_{\Psi} $,  if and only if $\Psi$ is conormal.
\end{lem}

We notice that the coassociativity of $\co_{\Psi}$ and the property of $\varepsilon_{\tn}$ to be a compatible counit are completely independent, in that there are examples of coassociative coproducts $\co_{\Psi}$ which do not admit $\varepsilon_{\ot}$ as counit (or even more in general which do not admit counit at all) and there are conormal twists $\Psi$ which do not satisfy the coassociativity condition \eqref{COeq}. \\

The class which has been studied in the literature is the one of conormal twists. Under the assumption of conormality of $\Psi$ the commutativity of the octagonal diagram  \eqref{COdiagr} can be equivalently reduced to the commutativity of the following pentagonal diagrams, by further composing both members of \eqref{COeq} with respectively $(\id_D\tn\varepsilon_C\tn\id_D\tn\id_C)$ and $(\id_D\tn\id_C\tn\varepsilon_D\tn\id_C)$:    
\begin{equation}
\label{CP1diagr}
\xymatrix @C=3pc{
C\tn D \ar[r]^-{\Psi} \ar[d]_-{\id_C\tn\co_D} & D\tn C \ar[r]^-{\co_D\tn\id_C} & D\tn D\tn C \\
C\tn D\tn D \ar[rr]^{\Psi\tn\id_D} & & D\tn C\tn D \ar[u]_-{\id_D\tn\Psi} 
}
\end{equation}
which amounts to 
\begin{equation}
\label{CP1eq}
(\id_D\tn\Psi)(\Psi\tn\id_D)(\id_C\tn\co_D) = (\co_D\tn\id_C)\Psi
\end{equation}
and
\begin{equation}
\label{CP2diagr}
\xymatrix @C=3pc{
C\tn D \ar[r]^-{\Psi} \ar[d]_{\co_C\tn\id_D} & D\tn C \ar[r]^-{\id_D\tn\co_C} & D\tn C\tn C \\
C\tn C\tn D \ar[rr]^-{\id_C\tn \Psi} & & C\tn D\tn C \ar[u]_-{\Psi\tn\id_C} 
}
\end{equation}
which amounts to 
\begin{equation}
\label{CP2eq}
(\Psi\tn\id_C)(\id_C\tn\Psi)(\co_C\tn\id_D) = (id_D\tn\co_C)\Psi \, .
\end{equation}

\begin{thm}
\label{thm-cae} 
Let $C,~D$ and $Y$ be coalgebras. The following two conditions are equivalent:
\begin{enumerate}
\item
There exists a coalgebra isomorphism $Y \simeq C \ot_\Psi D$ for some conormal twist map $\Psi$ solution of \eqref{COdiagr};
\item Y factorizes through $C$ and $D$, i.e. there exist coalgebra morphisms $u_C: Y \ra C$ and $u_D : Y \ra D $ such that the map $\eta:= (u_C \ot u_D)\Delta_Y : Y \ra C \ot D$ is an isomorphism of vector spaces.
\end{enumerate}
\end{thm}
The proof is based on the fact that if $C \ot_\Psi D$ is a coalgebra associated to a conormal twist $\Psi:C \ot D \ra D \ot C$ , the maps 
\begin{equation}
\label{projpi}
\begin{split}
\pi_C & : C \ot_\Psi D \ra C \, , \quad c \ot d \mapsto \varepsilon_D(d) \, c \\ 
\pi_D & : C \ot_\Psi D \ra D \, , \quad c \ot d \mapsto \varepsilon_C(c) \, d
\end{split}
\end{equation}    
are coalgebra morphisms, and the map 
\begin{equation}
\label{eta}
\eta:=(\pi_C \ot \pi_D)\Delta_\Psi :C \ot_\Psi D  \ra C \ot D
\end{equation} 
is an isomorphism of vector spaces. Indeed, it is  $\eta(c \ot d)= c \ot d$. The opposite implication is constructive with the conormal twist map given by $\Psi= (u_D \ot u_C) \circ \co_Y \circ \eta^{-1}$.\\

\noindent
A twisted coalgebra $C \ot_\Psi D$ enjoys the following universal property:   

\begin{prop} 
\label{univ-cae}
Let $C$ and $D$ be coalgebras, $\Psi$ a conormal twist  satisfying \eqref{COdiagr}. Let 
$Y$ be a coalgebra and $u_C: Y \ra C, ~ u_D:Y \ra D$ coalgebra morphisms such that 
\be
(u_D \ot u_C) \circ \Delta_Y = \Psi \circ (u_C \ot u_D) \circ \Delta_Y
\ee
then there exists a unique coalgebra morphism $\omega:Y \ra C \ot_\Psi D$ such that the following diagram commutes
\begin{equation}
\label{univ-diagr}
\xymatrix{
  & C\ot_{\Psi}D \ar[dl]_{\pi_C}\ar[dr]^{\pi_D} &   \\
C &                                         & D \\
  & Y \ar[ul]^{u_C}\ar[ur]_{u_D}\ar@{.>}[uu]^{\omega} &
}
\end{equation}
\end{prop}


\section{Non-conormal twists}\label{sec-nct}

We are interested in extending the study of solutions of the coassociativity condition \eqref{COeq} to the non-conormal case. We start by generalizing the notion of conormality.

\begin{defi}
Let  $Z : C\ot D \ra D\ot C \, , ~  c \ot d \mapsto d^{[Z]} \ot c^{[Z]}$ be a bilinear map. Set $\varepsilon^{\tau}_{\tn}:=\varepsilon_D\tn\varepsilon_C$. A twist map $\Psi$ is said to be right $Z$-conormal if it satisfies 
\begin{equation}
\label{rZcon}
\varepsilon({(d^{[\Psi]}})^{[Z]})\varepsilon({\uno{c}}^{[Z]}) {\due{c}}^{[\Psi]}= \varepsilon(d) c   \, 
\end{equation}
for all $c \in C, ~ d \in D$. Similarly, it is said to be left $Z$-conormal if
\begin{equation}
\label{lZcon}
\varepsilon({(c^{[\Psi]}})^{[Z]})\varepsilon({\due{d}}^{[Z]}) {\uno{d}}^{[\Psi]}= \varepsilon(c) d  
\end{equation}
for all $c \in C, ~ d \in D$. It is said to be $Z$-conormal if it is both left and right $Z$-conormal.
\end{defi}

\begin{lem}
Let $\Psi$ be a twist map.
\begin{enumerate}[(i)] 
\item  $\Psi$ is left (resp. right) conormal if and only if $\Psi$ is left (resp. right) $\tau$-conormal.
\item If $\Psi$ is left (resp. right) conormal, then $\Psi$ is  left (resp. right)  $\Psi$-conormal.
\item If  $\Psi$ satisfies conditions \eqref{CP1eq}\eqref{CP2eq}, then  $\Psi$ is  left (resp. right) $\Psi$-conormal if and only if 
$\Psi$ is  left (resp. right) conormal.
\end{enumerate}
\end{lem}

\begin{proof} The first point is trivial. In the second, suppose $\Psi$ is conormal. Then by applying twice \eqref{rcon} we get
\begin{equation*}
\begin{split}
(\varepsilon_D\tn\varepsilon_C\tn\id_C)(Z\tn\id_C)(\id_C\tn\Psi)(\co_C\tn\id_D) 
& = (\varepsilon_C\tn\varepsilon_D\tn\id_C)(\id_C\tn\Psi)(\co_C\tn\id_D) \\
& = (\varepsilon_C\tn\varepsilon_D)(\co_C\tn\id_D) = \id_C\tn\varepsilon_D
\end{split}
\end{equation*} 
i.e.  $\Psi$ is right $\Psi$-conormal. Left $\Psi$-conormality is completely analogous. For the last point, suppose that $\Psi$ satisfies  \eqref{CP2eq}. Then
\begin{equation*}
\begin{split}
(\varepsilon_D\tn\varepsilon_C\tn\id_C)(\Psi\tn\id_C)(\id_C\tn\Psi)(\co_C\tn\id_D) 
& = (\varepsilon_D\tn\varepsilon_C\tn\id_C)(\id_D\tn\co_C)\Psi \\
& = (\varepsilon_D\tn\id_C)\Psi
\end{split}
\end{equation*}
so that right $Z$-conormality implies right conormality. For the left-version we use \eqref{CP1eq}. 
\end{proof}

\begin{lem}
Consider $Z$, $Z': C \ot D \ra D \ot C$. Then the following two conditions are equivalent: \begin{enumerate}[(i)]
\item a twist map $\Psi$ is $Z$ conormal if and only if it is $Z'$ conormal 
\item $Im(Z-Z') \in ker(\varepsilon^{\tau}_\ot)$
\end{enumerate}
\end{lem} 

\begin{prop}
\label{prop-counita}
A twist map $\Psi$ which satisfies \eqref{COeq} is $Z$-conormal for some $Z$ if and only if 
\be
\varepsilon_{Z}:= \varepsilon^{\tau}_{\tn} \circ Z= (\varepsilon_D\tn\varepsilon_C)\circ Z
\ee 
defines a counit for the twisted coproduct $\co_{\Psi}$.
\end{prop}

\begin{proof}
Assume $\Psi$ is $Z$-conormal, so in particular it is left $Z$-conormal. Then
\begin{equation*}
\begin{split}
(\id_{C\tn D}\tn\varepsilon_Z)\co_{\psi} 
& = (\id_{C\tn D}\tn\varepsilon_Z)(\id_C\tn\Psi\tn\id_D)(\co_C\tn\co_D) \\
& = \left( \id_C\tn \big( (\id_D\tn\varepsilon_Z)(\Psi\tn\id_D)(\id_C\tn\co_D) \big) \right) (\co_C\tn\id_D) \\
& = (\id_C\tn\varepsilon_C\tn\id_D)(\co_C\tn\id_D) = \id_{C\tn D} \, .
\end{split}
\end{equation*}
Similarly, by using the right $Z$-conormality we prove $(\varepsilon_Z\tn\id_{C\tn D})\co_{\Psi} = \id_{C\tn D}$. On the other side, suppose $\varepsilon_Z$ is a counit for $\co_{\Psi}$. Then $ (\id_{C\tn D}\tn\varepsilon_Z)\co_{\psi} = \id_{C\tn D}$ and by applying $\varepsilon_C\tn\id_D$ to both sides we get the left $Z$-conormality condition \eqref{lZcon}. By applying $\id_C\tn\varepsilon_D$ we get the right $Z$ conormality condition \eqref{rZcon}. 
\end{proof}

In the following, when not necessary to specify the map $Z$, we will use the notation $\varepsilon_\Psi$ to indicate the counit of a given $Z$-conormal twist $\Psi$. Note that it is not restrictive to assume a counit of the form $\varepsilon_Z=\varepsilon^{\tau}_{\tn}\circ Z$. Indeed, given a generic counit $\varepsilon_{\Psi}$, we can always find a $Z$ such that $\varepsilon_{\Psi}=\varepsilon_Z$, for example $Z(c\otimes d)=\varepsilon_{\Psi}(c_{(1)}\otimes d_{(2)}) \, d_{(1)}\otimes c_{(2)}$.


\subsection{Universal properties and factorization}\label{sec-univ}

In this section, if not otherwise stated, we consider generic twist maps $\Psi$ such that the (non necessarily coassociative) coproduct $\co_{\Psi}$ admits compatible counit $\tco$. With an abuse of terminology sometimes we refer to $\tco$ as a compatible counit for $\Psi$, rather than for  $\co_\Psi$.  Further properties, for example that $\Psi$ satisfies the coassociativity condition \eqref{COeq}, are explicitly mentioned when relevant.  \\

Consider the projection maps $p_C:C\tn_{\Psi}D \ra C \, , ~ p_D:C\tn_{\Psi}D \ra D$ defined respectively by
\begin{equation}
\label{proj}
\begin{split}
p_C & :=(\id_C\tn\tco)(\co_C\tn\id_D) \, , \quad p_C(c\tn d) = c_{(1)} \tco(c_{(2)}\tn d) \, ;\\
p_D & :=(\tco\tn\id_D)(\id_C\tn\co_D) \, , \quad p_D(c\tn d) = \tco(c\tn d_{(1)}) \, d_{(2)} \, .
\end{split}
\end{equation}

\begin{lem} 
The projection maps $p_C$ and $p_D$ are coalgebra morphisms.
\end{lem}
\begin{proof} 
We have to show that $\varepsilon_C \circ p_C = \tco$ and $(p_C \tn p_C) \circ \co_\Psi = \co_C \circ p_C$. The first statement is trivial.   
For the second one we have
\begin{eqnarray*}
(p_C \tn p_C) \circ \co_\Psi (c \tn d)&=& (p_C \tn p_C)(\uno{c} \tn (\uno{d})^{[\Psi]} \tn (\due{c})^{[\Psi]} \tn \due{d}) \\
 & = & (\id_C\tn p_C) (\uno{c} \tn \tco(\due{c}\tn (\uno{d})^{[\Psi]}) \tn (\tre{c})^{[\Psi]} \tn \due{d} ) \\
 & = & (\id_C\tn p_C) (\uno{c} \tn \due{c}\tn d) \\
 & = & \uno{c} \tn \due{c} \tco(\tre{c}\tn d) \\
 & = & \co_C \circ p_C (c\tn d) 
\end{eqnarray*} 
The proof for $p_D$ is completely analogous.  
\end{proof}

These maps generalize the projections $\pi_C,\pi_D$ \eqref{projpi} to generic twists; indeed for $\Psi$ conormal we recover $p_C=\pi_C$ and $p_D=\pi_D$. 

Let $(C\tn D)'$ be the set of linear functionals $\phi:C\tn D\ra\kk$. We consider two different algebra structures on $(C\tn D)'$; given $\phi ,\phi' \in (C\tn D)'$ we can multiply them via the convolution product
\begin{equation}
\label{convpr}
(\phi\ast\phi')(c\tn d) = \phi(c_{(1)}\tn d_{(1)}) \, \phi'(c_{(2)}\tn d_{(2)}) 
\end{equation} 
or via a $\star$-product defined as
\begin{equation}
\label{starpr}
(\phi\star\phi')(c\tn d) = \phi(c_{(2)}\tn d_{(1)}) \, \phi'(c_{(1)}\tn d_{(2)}) \, .
\end{equation}
Both algebras $((C\tn D)',\ast )$ and $((C\tn D)',\star )$ are associative, with unit element the tensor counit $\varepsilon_{\tn}$. Indeed  $((C\tn D)',\star )= ((C^{cop}\tn D)',\ast )$, i.e. the $\star$ product is just the convolution product on $C^{cop} \ot D'$. Then if $C$ is co-commutative it is
$((C\tn D)',\star ) = ((C\tn D)',\ast )$, while if $D$ is co-commutative $(C^{cop}\tn D)^{cop}=(C\tn D^{cop}) = (C\tn D)$ implies that $( (C\tn D)',\star ) = ( (C\tn D)', \ast^{op} )$. \\

\begin{prop}
\label{invmu} 
Let $\Psi$ be a twist with compatible counit $\tco$. Let $p_C,p_D$ be the coalgebra maps defined in \eqref{proj}. The counit $\tco$ is invertible in $[(C\tn D)', \star]$ if and only if the map $\mu:C\tn_{\Psi} D\ra C\tn D$, 
$$ \mu(c\tn d):=(p_C\tn p_D)\co_{\Psi} (c\tn d) = \tco(\due{c}\tn\uno{d}) \, \uno{c}\tn\due{d} \, ,$$ 
is an isomorphism of vector spaces. 
\end{prop}

\begin{proof} We first assume that $\tco$ is invertible in  $[(C\tn D)', \star]$. Denote $\tco^{\star}$ the inverse of $\tco$ with respect to the product $\star$. We claim that
$$ \mu^{-1}(c\tn d) = \tco^{\star}(\due{c}\tn \uno{d})\, \uno{c}\tn\due{d} \, .$$
By a direct computation
\begin{equation*}
\begin{split}
\mu^{-1}\circ\mu (c\tn d) & = \mu^{-1} ( \tco(\due{c}\tn\uno{d}) \uno{c}\tn\due{d} ) \\
 & =  \tco( \tre{c}\tn\uno{d} \tco^{\star}(\due{c}\due{d}) \uno{c}\tn\tre{d} \\
 & = (\tco\star\tco^{\star})(\due{c}\tn\uno{d}) \uno{c}\tn\due{d} =  c\tn d \, ; \\
\mu\circ\mu^{-1} (c\tn d) & = \mu( \tco^{\star}(\due{c}\tn\uno{d}) \uno{c}\tn\due{d} ) \\
 & = \tco^{\star}(\tre{c}\tn\uno{d} \tco(\due{c}\tn\due{d})) \uno{c}\tn\tre{d} \\
 & = (\tco^{\star}\star\tco)(\due{c}\tn\uno{d}) \uno{c}\tn\due{d} = c\tn d \, . 
\end{split}
\end{equation*}
For the opposite implication, suppose $\mu$ is invertible. Then we show that $\tco$ is invertible with inverse
$\tco^\star:= \varepsilon_\ot \circ \mu^{-1}$ . On the one hand
$$
(\tco \star  \varepsilon_\ot \circ \mu^{-1}) (c\tn d)  =  (\varepsilon_\ot \circ \mu^{-1}) \left(  \tco (\due{c} \ot \uno{d}) \, \uno{c}\tn\due{d} \right) =  (\varepsilon_\ot \circ \mu^{-1}) (\mu(c \tn d))=  \varepsilon_\ot (c\tn d) \, ,
$$
so that $\tco^\star$ is the right inverse of $\tco$.
Using this result we can prove that  $\alpha(c \ot d):= \tco^\star (\due{c} \ot \uno{d}) \, \uno{c} \ot \due{d}$ coincides with $\mu^{-1}$. Indeed 
 \begin{eqnarray*}
(\alpha \circ \mu) (c \ot d)&=& \tco (\due{c} \ot \uno{d}) \alpha (\uno{c} \ot \due{d}) = \tco (\tre{c} \ot \uno{d}) \tco^\star (\due{c} \ot \due{d})\, \uno{c} \ot \tre{d}\\ &=& (\tco \star \tco^\star) (\due{c} \ot \uno{d}) \uno{c} \ot \due{d} = c \ot d
 \end{eqnarray*}
hence $\alpha$ is the left inverse of $\mu$, and therefore $\alpha= \mu^{-1}$ being $\mu$ invertible.
We can now prove that $\tco^\star$ is the left inverse of $\tco$: 
$$
(\tco^\star \star \tco ) (c\tn d)  = \tco \left( \uno{c} \ot \due{d} \tco^\star (  \due{c}\tn\uno{d} ) \right)=  \tco(\mu^{-1}(c \ot d))=\varepsilon_\ot (c\tn d) \, 
$$ 
since $\tco= \varepsilon_\ot \circ \mu$.
\end{proof}

Note that the above result holds for generic twists, not necessarily solutions of $\eqref{COeq}$.
\\

\noindent As a consequence of Thm. \ref{thm-cae} we have the following
\begin{cor}\label{prtr}
Let $(C \ot_\Psi D, \co_\Psi, \tco)$ be a twisted coalgebra associated to a twist map $\Psi$ which satisfies the coassociativity condition \eqref{COeq}. 
Suppose there exist two coalgebra morphisms $p_C: C \ot_\Psi D \ra C$, $p_D: C \ot_\Psi D \ra D$ such that the map $\mu=(p_C\tn p_D)\co_{\Psi}$ is invertible. Then $\widetilde{\Psi}:=(p_D\tn p_C)\co_{\Psi}\mu^{-1}$ is a conormal twist, and $\mu:C\tn_{\Psi}D\ra C\tn_{\widetilde{\Psi}}D$ realizes a coalgebra isomorphism.
\end{cor}

We have the following universal characterization of the twisted coalgebra $C\ot_\Psi D$ (cf. Prop. \ref{univ-cae}).

\begin{thm}\label{thm-univ-nct}
Let $C\tn_{\Psi} D=(C\tn D,\co_{\Psi},\tco)$ be a twisted coalgebra for some twist $\Psi$ solution of $\eqref{COeq}$, and $p_C,p_D$ the projections introduced in \eqref{proj}. Let $Y$ be a coassociative coalgebra, together with coalgebra morphisms $u_C:Y\ra C \, , \, u_D:Y\ra D$ such that
\begin{equation}
\label{twcond}
( \, u_D \ot u_C\, ) \circ \co_Y = \Psi \circ ( \, u_C \ot u_D \, ) \circ \co_Y \, , \qquad 
\tco \, ( \, u_C \ot u_D\, ) \co_Y = \varepsilon_Y \, .
\end{equation}
Then there exists a coalgebra morphism $\omega : Y\ra C\ot_{\Psi}D$ such that $u_C=p_C\circ\omega$ and $u_D=p_D\circ\omega$, i.e. the following diagram commutes
\begin{equation}
\label{cdu}
\xymatrix{
  & C\ot_{\Psi}D \ar[dl]_{p_C}\ar[dr]^{p_D} &   \\
C &                                         & D \\
  & Y \ar[ul]^{u_C}\ar[ur]_{u_D}\ar@{.>}[uu]^{\omega} &
}
\end{equation}
\end{thm}

\begin{proof}
Set $\omega:= ( \, u_C\ot u_D \, ) \circ \co_Y$. The proof that $\omega$ is a coalgebra map is the same of Corollary \ref{prtr} for the map $\mu$: for $y\in Y$ we have 
\begin{equation*}
\begin{split}
(\,\co_{\Psi}\circ \omega \,) (y) & = \co_{\Psi} ( \, u_C(y_{(1)}) \ot u_D(y_{(2)}) \, ) \\
 & = (1\ot\Psi\ot 1) \circ (\co_C\ot\co_D) ( \, u_C(y_{(1)}) \ot u_D(y_{(2)}) \, ) \\
 & = (1\ot\Psi\ot 1) ( \, u_C(y_{(1)}) \ot u_C(y_{(2)}) \ot u_D(y_{(3)}) \ot u_D(y_{(4)}) \, ) \\
 & = u_C(y_{(1)}) \ot \Psi \big( \, (u_C\ot u_D)\circ\co_Y (y_{(2)}) \, \big) \ot u_D(y_{(3)}) \\
 & = u_C(y_{(1)}) \ot ( u_D\ot u_C)\circ\co_Y (y_{(2)}) \ot u_D(y_{(3)}) \\
 & = u_C(y_{(1)}) \ot u_D(y_{(2)}) \ot u_C(y_{(3)}) \ot u_D(y_{(4)}) \\
 & = (\omega\ot\omega)\circ\co_Y (y) \, .
\end{split}
\end{equation*}
The second condition in \eqref{twcond} now reads $\tco\circ\omega = \varepsilon_Y$, which is the intertwining property for the counit. We finally show that $\omega$ indeed makes the diagram \eqref{cdu} commutative. We compute $p_C\circ\omega (y)$ for a generic $y \in Y$, using the fact that $\omega$ is a coalgebra map:
\begin{equation*}
\begin{split}
p_C\circ\omega \, (y) & = p_C\circ (u_C\ot u_D) \circ \co_Y (y) \\
 & = (1\ot \tilde{\varepsilon} ) \circ (\co_C\ot 1) ( \, u_C(y_{(1)})\ot u_D(y_{(2)}) \,) \\
 & = (1\ot \tilde{\varepsilon} ) ( \, u_C(y_{(1)}) \ot u_C(y_{(2)}) \ot u_D(y_{(3)}) \, ) \\
 & = u_C(y_{(1)}) \, \tilde{\varepsilon} ( \, \omega (y_{(2)})\, ) \\
 & = u_C(y_{(1)}) \, \varepsilon_Y (y_{(2)}) = u_C (y) \, .
\end{split}
\end{equation*}
The computation for $u_D= p_D\circ\omega(y)$ is similar.
\end{proof}

\begin{rem} 
In the above Theorem the coalgebra map $\omega$ is unique for twists $\Psi$ such that $\mu=(p_C\otimes p_D)\co_{\Psi}$ is injective.
Indeed suppose there exist two coalgebra maps $\omega \, , \omega'$ which make the diagram \eqref{cdu} commute. Then consider the coalgebra map $\Omega = \omega - \omega'$; its image is a sub-coalgebra in $C\ot_{\Psi}D$, on which both $p_C$ and $p_D$ vanish. Hence $\mu(x)=0$ for every $x \in im(\Omega)$, so it must be $\Omega =0$.
\end{rem}

\begin{rem}
Suppose $Y$ is itself a twisted coalgebra, $Y=C\otimes_{\chi} D$ for some twist $\chi$ solution of \eqref{COeq} admitting counit, and the map $\mu_{\chi}=(u_C\otimes u_D)\co_Y$ is invertible (i.e. the case of Corollary \ref{prtr}).  Then condition \eqref{twcond} is satisfied by 
the conormal twist $\Psi= \widetilde{\chi} = (u_D\otimes u_C)\co_Y \mu_{\chi}^{-1}$ and the coalgebra isomorphism $\omega:C\otimes_{\widetilde{\chi}}D\ra C\otimes_{\chi} D$ is the same of Corollary \ref{prtr}.
\end{rem}

\section{Twists from functionals}\label{se-functional}

In this section we describe a class of solutions of condition \eqref{COeq}, which are intrinsically non conormal. In particular this shows that once we drop the request of conormality, conditions \eqref{CP1eq}\eqref{CP2eq} are no longer necessary for the coassociativity of $\co_{\Psi}$. 

The basic idea is to solve \eqref{COeq} by moving the map $\Psi$ to the left in both members; this is possible if we ask $\Psi$ to intertwine suitable coactions of $C$ and $D$. We formalize this property in a categorical language as follows.

We denote with $\M$ the category of left $D$ - right $C$ comodules which have commuting left and right coactions; we denote by $Obj(\M)$ its objects and by $Mor(\M)$ its morphisms. We remember that as long as $C,D$ are only coalgebras and not bialgebras, $\M$ does not have a monoidal structure. \\
We observe that the map 
 $$ _D\Phi = (\co_D \ot 1): D\ot C\ra D\ot D\ot C $$
defines a left coaction of $D$ on $D\ot C$, and that similarly, the map 
$$
\Phi_C = (1\ot \co_C): D\ot C\ra D\ot C\ot C
$$
defines a right $C$ coaction of $C$ on $D\ot C$.
These coactions $_D\Phi$ and $\Phi_C$ commute.
Hence $D\tn C \in Obj(\M)$. The coalgebra $C\tn D$ is an object in $\M$ too: 
\begin{equation*}
_D\Phi^{\tau} = (1\ot\tau)\circ\Phi_D\circ\tau= (1\ot\tau)\circ(\co_D \ot 1)\circ\tau : C\ot D\ra D\ot C\ot D
\end{equation*}
defines a left coaction of $D$ on $C\ot D$ and the map 
\begin{equation*}
\Phi^{\tau}_C = (\tau\ot 1)\circ\Phi_C\circ\tau = (\tau\circ 1)\circ(1\ot \co_C)\circ\tau: C\ot D\ra C\ot D\ot C
\end{equation*}
defines a right $C$ coaction of $C$ on $D\ot C$.
The two coactions $_D\Phi^{\tau}$ and $\Phi^{\tau}_C$ commute.
\\

\begin{thm}
\label{twistmorph}
Any twist map $\Psi \in Mor(\M)$ is a solution of \eqref{COeq}, hence it defines a coassociative coproduct $\co_{\Psi}$ on $C\ot D$.
\end{thm}

\begin{proof} 
In the right-hand side of \eqref{COeq} we can use the fact that $\Psi$ is a morphism of left $D$-comodules, i.e. $_D\Phi\circ\Psi = (\id_D\ot \Psi)\circ {_D\Phi}^{\tau}$, to write 
$$ (\id_C\ot \co_D \ot \id_C)(\id_C\ot \Psi) = (\id_C\ot {_D\Phi})(\id_C\ot \Psi) = (\id_C\ot\id_D\ot \Psi)(\id_C\ot {_D\Phi}^{\tau}) \, .$$
Similarly in the left-hand side we can use the hypothesis that $\Psi$ is a morphism of right $C$-comodules, i.e. $\Phi_C\circ\Psi = (\Psi\ot\id_C)\circ \Phi^{\tau}_C$ , to write
$$ (\id_D\ot\co_C\ot\id_D)(\Psi\ot\id_D) = (\Phi_C\ot\id_D)(\Psi\ot\id_D) = (\Psi\ot\id_C\ot\id_D)(\Phi^{\tau}_C\ot\id_D)  \, .$$
Hence  \eqref{COeq} is satisfied provided 
$$ (\Psi\ot\Psi)\circ (\Phi^{\tau}_C\ot\id_D)\circ (\id_C\ot \co_D) = (\Psi\ot \Psi) \circ (\id_C\ot {_D\Phi}^{\tau}) \circ (\co_c\ot\id_D) \, $$
and in fact this equation holds because
$$ (\Phi^{\tau}_C\ot\id_D) (\id_C\ot \co_D) = (\id_C\ot {_D\Phi}^{\tau}) (\co_c\ot\id_D) = \co_{\ot} $$
as one can easily verify on a generic element $c\ot d \in C\tn D$.
\end{proof}\bigskip

It is useful to notice that the property $\Psi \in Mor(\M)$ can be equivalently expressed in terms of the map  $\Psi^{\prime}=\Psi\circ\tau: D\ot C\ra D\ot C$ that will be also referred to as a twist in the following. It turns out that $\Psi $ is a morphism in $\M$ if and only if $\Psi^{\prime}$ does, that is if and only if 
\begin{align}
\label{LC1}
(\co_D\tn\id_C)\Psi^{\prime} & = (\id_D\tn\Psi^{\prime})(\co_D\tn\id_C) \\
\label{LC2}
(\id_D\tn\co_C)\Psi^{\prime} & = (\Psi^{\prime}\tn\id_C)(\id_D\tn\co_C) \, .
\end{align} 
\medskip
This allows to conclude easily that
\begin{cor}\label{cor-comp}
The space of solutions of \eqref{LC1} and \eqref{LC2} is a unital algebra with multiplication given by the composition of morphisms.
\end{cor}
We observe that conversely, it is not clear (at least to us) whether the composition of two twists which are solutions of \eqref{COdiagr} still gives a solution of the same equation. The same question remains open for  the  solutions of \eqref{CP1diagr} and \eqref{CP2diagr}.\\

Furthermore, we get a simplified version of the above conditions   \eqref{LC1} and \eqref{LC2} by applying respectively $\id_D\tn\varepsilon_D\tn\id_C$ and $\id_D\tn\varepsilon_C\tn\id_C$ to them. On the generic element $d\tn c\in D\tn C$ we have
\begin{align}
\label{eLC1}
d^{[\Psi]}\tn c^{[\Psi]} & = d_{(1)}\varepsilon_D\big( (d_{(2)})^{[\Psi]} \big) \tn c^{[\Psi]} \\
\label{eLC2}
d^{[\Psi]}\tn c^{[\Psi]} & = d^{[\Psi]} \tn \varepsilon_C \big( (c_{(1)})^{[\Psi]} \big) \, c_{(2)} \, .
\end{align}
The previous identities  are indeed equivalent to \eqref{LC1} and \eqref{LC2} (to prove the opposite implication it is enough to apply $\Delta_C \tn id$, resp $id \tn \Delta_D$ to \eqref{eLC1}, resp \eqref{eLC2}) and can be used to simplify the expression of the twisted coproduct as
\begin{equation}
\begin{split}
\label{cot}
\co_{\Psi}(c\tn d) & = c_{(1)}\tn (d_{(1)})^{[\Psi]} \tn (c_{(2)})^{[\Psi]} \tn d_{(2)} \\
 & = \varepsilon_C\left( (c_{(2)})^{[\Psi]} \right) \, \varepsilon_D\left( (d_{(2)})^{[\Psi]} \right) \, c_{(1)}\tn d_{(1)} \tn c_{(3)} \tn d_{(3)} \, .
\end{split}
\end{equation}
This shows that the coproduct is only `mildly twisted'. Nevertheless, apart from the trivial case $\Psi=\tau$, this deformation turns out to be intrinsically non-conormal, as the next Lemma shows.

\begin{lem}
\label{lem-conorm}
Consider a twist $\Psi \in Mor(\M)$. Then $\Psi$ is conormal if and only if  $\Psi$ is the flip map $\tau$.
\end{lem}

\begin{proof}
Suppose that $\Psi$ is conormal. Then by using  \eqref{eLC1} we have 
$$ \Psi(c \ot d)= d^{[\Psi]} \ot c^{[\Psi]}=  d_{(1)} \ot \varepsilon_D \left(  (d_{(2)})^{[\Psi]} \right) \ot c^{[\Psi]}= d_{(1)} \varepsilon_D(d_{(2)}) \ot c =  d\tn c \, .$$ The opposite implication is trivial. 
\end{proof}

\begin{rem}
The possibility to compose twists $\Psi':D\tn C\ra D\tn C$, $\Psi'\in Mor(\M)$, could raise the following question about the need to consider non-conormal twists (and therefore deformed counits). Given a twisted coalgebra $(C\tn_\Psi D, \co_{\Psi},\tco)$, one might wonder whether it is possible to find a second twist such that the composition of the two is conormal (i.e. we restore $\varepsilon_{\tn}$ as the compatible counit). A positive answer would implies that we can always, up to some `gauge equivalence' (the second twist), consider conormal twists without loss of generality. The previous Lemma shows that  this is not the case for the class of twists we are considering. Since the only conormal twist in $Mor(\M)$ is the trivial one $\Psi=\tau$, the only way to `untwist' the counit is by composing it with the inverse twist, thus `untwisting' the coproduct as well.
\end{rem}
\bigskip

We proceed now with the study of this class of solutions of \eqref{COdiagr} by showing  that there is a one to one correspondence between the twists  which are  solutions of \eqref{LC1}\eqref{LC2} and the functionals 
 on $C\tn D$. This provides an easier description of the resulting twisted coalgebra, expressing several properties (existence of counit, composition of twists, universal properties etc) in terms of the functionals themselves. 
\\
 In the remaining of the section we will denote by $\tw$ the set of twist maps $\Psi':D\tn C\ra D\tn C$ such that $\Psi'\in Mor(\M)$.
\\
As already observed before (Corollary \ref{cor-comp}) the set $\tw$ becomes an algebra $(\tw,\circ)$ with the composition $\circ$ of morphisms in  $Mor(\M)$.  

\begin{thm}\label{prop-F}
There is an algebra isomorphism $F:(\tw,\circ)\ra \big((C\tn D)',\star\big)$ sending a twist map $\Psi'$ into the associated functional 
$$F(\Psi')= \phi_{\Psi} := \varepsilon_{\tn}^{\tau}\circ\Psi'\circ\tau = (\varepsilon_D\tn\varepsilon_C)\circ\Psi \, .$$
\end{thm}
\begin{proof} 
We first show that $F$ is invertible. Given $\phi\in (C\tn D)'$, set for all $c \in C,~d \in D$
$$(F^{-1}(\phi))(d\tn c)=\Psi'_{\phi}(d\tn c):=\phi(c_{(1)}\tn d_{(2)})\, d_{(1)}\tn c_{(2)} \, .$$
The map $\Psi'_{\phi}$ satisfies the conditions \eqref{LC1}\eqref{LC2}, namely $\Psi'_{\phi}\in \tw$; indeed
$$ (\co_D\tn\id_C)\Psi'_{\phi}(d\tn c) =  \phi(c_{(1)}\tn d_{(2)}) \,  (\co_D\tn\id_C)(d_{(1)}\tn c_{(2)}) = \phi(c_{(1)}\tn d_{(3)}) \,  d_{(1)}\tn d_{(2)}\tn c_{(2)} $$
and
$$ (\id_D\tn\Psi'_{\phi})(\co_D\tn\id_C)(d\tn c) = (\id_D\tn\Psi'_{\phi})(d_{(1)}\tn d_{(2)}\tn c) = \phi(c_{(1)}\tn d_{(3)}) \, d_{(1)}\tn d_{(2)}\tn c_{(2)} \, .$$
Similarly
$$ (\id_D\tn\co_C)\Psi'_{\phi}(d\tn c) =  \phi(c_{(1)}\tn d_{(2)}) \, (\id_D\tn\co_C)(d_{(1)}\tn c_{(2)}) =  \phi(c_{(1)}\tn d_{(2)}) \, d_{(1)}\tn c_{(2)}\tn c_{(3)} $$
and
$$ (\Psi'_{\phi}\tn\id_C)(\id_D\tn\co_C)(d\tn c) = (\Psi'_{\phi}\tn\id_C)(d\tn c_{(1)}\tn c_{(2)}) = \phi(c_{(1)}\tn d_{(2)}) \, d_{(1)}\tn c_{(2)}\tn c_{(3)} \, . $$
We verify that $FF^{-1}=F^{-1}F=\id$:
$$ \phi_{(\Psi_{\phi})}(c\tn d) = \varepsilon_{\tn}^{\tau}(\Psi_{\phi})(c\tn d)) = \phi(c_{(1)}\tn d_{(2)})\, \varepsilon_D(d_{(1)}) \,  \varepsilon_C(c_{(2)}) = \phi (c\tn d) $$
and similarly 
$$ \Psi'_{(\phi_{\Psi})}(d\tn c) = \phi_{\Psi}(c_{(1)}\tn d_{(2)}) \, d_{(1)}\tn c_{(2)} = \varepsilon_D((d_{(2)})^{[\Psi]}) \varepsilon_C((c_{(1)})^{[\Psi]}) \, d_{(1)}\tn c_{(2)} = \Psi'(d\tn c) $$
where in the last equality we used \eqref{eLC1} and \eqref{eLC2}. The maps $F$ and $F^{-1}$ are linear. The compatibility with the algebra structure follows from $F(\id_{D\tn C}) = \phi_{(\id_{D\tn C})} = \varepsilon_{\tn}$ and  
\begin{equation*}
\begin{split}
\phi_{(\Psi_2 \circ \tau \circ\Psi_1)}(c\tn d) & = \varepsilon_{\tn}^{\tau} \big( (d^{[\Psi_1]})^{[\Psi_2]} \tn (c^{[\Psi_1]})^{[\Psi_2]} \big) \\
(\phi_{\Psi_2}\star\phi_{\Psi_1}) (c\tn d) & = \phi_{\Psi_2}(c_{(2)}\tn d_{(1)}) \, \phi_{\Psi_1} (c_{(1)} \tn d_{(2)}) \\
 & = \varepsilon_{\tn}^{\tau} \big( (d_{(1)})^{[\Psi_2]} \tn (c_{(2)})^{[\Psi_2]} \big) \, \varepsilon_{\tn}^{\tau} \big( (d_{(2)})^{[\Psi_1]} \tn (c_{(1)})^{[\Psi_1]} \big) \\
 & = (\varepsilon_{\tn}^{\tau}\circ\Psi_2') \big( \varepsilon_D( (d_{(2)})^{[\Psi_1]} ) \, d_{(1)} \tn \varepsilon_C ( (c_{(1)})^{[\Psi_1]} ) \, c_{(2)} \big) \\
 & = (\varepsilon_{\tn}^{\tau}\circ\Psi_2') (d^{[\Psi_1]}\tn c^{[\Psi_1]} ) = \varepsilon_{\tn}^{\tau}\big( (d^{[\Psi_1]})^{[\Psi_2]} \tn (c^{[\Psi_1]})^{[\Psi_2]} \big) \, .
\end{split}
\end{equation*} 

\end{proof}\medskip

The previous result is a particular instance of the Hom-tensor relations in $\M \simeq ~{^{D \ot C^{cop}}M}$. For any
$\kk$-coalgebra $P$, $N \in ~{^{P}M} $, $V$ a $\kk$-vector space, there exists an isomorphism of $\kk$-vector spaces (see e.g. \cite[\S 3.10]{brz-book}):
$$
Mor(N, P \ot V) \stackrel{\simeq}{\longrightarrow} Hom_{\kk} (N, V) \; 
$$ 
which becomes an algebra isomorphism for $N=P$ and $V = \kk$.\\
\bigskip

\noindent
The expression of the twisted coproduct $\co_{\Psi}$ in \eqref{cot} can now be rewritten as
\begin{equation}
\label{cot2}
\co_{\Psi}(c\tn d) = \phi_{\Psi}(c_{(2)}\tn d_{(2)}) \, c_{(1)}\tn d_{(1)} \tn c_{(3)} \tn d_{(3)} \, .
\end{equation}
In the following, whenever we want to emphasize the role of the functional $\phi$, we write $\co_\phi$ for the twisted coproduct $\co_{\Psi_\phi}$.

\begin{prop}
\label{counfun}
Let $\Psi' \in \tw$. The twisted coproduct $\co_{\Psi}$ admits a counit if and only if $\phi_{\Psi}$ is invertible with respect to the convolution product, and in that case $\phi_{\Psi}^{-1}$ is the twisted counit.
\end{prop}

\begin{proof}
Suppose $\co_{\Psi}$ admits counit $\varepsilon_{\Psi}$. Then from $(\varepsilon_{\Psi}\tn\id_{C\tn D})\co_{\Psi} = (\id_{C\tn D}\tn\varepsilon_{\Psi})\co_{\Psi} = \id_{C\tn D}$ we get respectively
\begin{equation*}
\begin{split}
(\varepsilon_{\Psi}\tn\id_{C\tn D})\co_{\Psi} (c\tn d) & = \varepsilon_{\Psi}(c_{(1)}\tn d_{(1)}) \, \phi_{\Psi}(c_{(2)}\tn d_{(2)}) \, c_{(3)}\tn d_{(3)} \\
 & = (\varepsilon_{\Psi}\ast\phi_{\Psi})(c_{(1)}\tn d_{(1)}) \, c_{(2)}\tn d_{(2)} \, , \\
(\id_{C\tn D}\tn\varepsilon_{\Psi})\co_{\Psi} (c\tn d) & = c_{(1)}\tn d_{(1)} \, \varepsilon_{\Psi}(c_{(3)}\tn d_{(3)}) \, \phi_{\Psi}(c_{(2)} \tn d_{(2)}) \\
 & = c_{(1)}\tn d_{(1)} \, (\phi_{\Psi}\ast \varepsilon_{\Psi})(c_{(2)}\tn d_{(2)}) \, . 
\end{split}
\end{equation*}
This shows that $\varepsilon_{\Psi}$ exists if and only if $(\phi_{\Psi})^{-1}$ does, and in this case $\varepsilon_{\Psi}=(\phi_{\Psi})^{-1}$.
\end{proof}

As a corollary of this result we get a different proof of Lemma \ref{lem-conorm}: suppose  $\Psi$ conormal, i.e. $(\phi_{\Psi})^{-1}=\varepsilon_{\tn}$, then $\phi_{\Psi}=\varepsilon_{\tn}$ and hence $\Psi=\tau$. 

\begin{rem}
In the above, the role of the two functionals $\phi_\Psi$ and $\tco$ is completely symmetric: we have a second twisted coalgebra where $\phi^{-1}$ is used to deform the coproduct and $\phi$ is the compatible counit. \\
\end{rem}
\medskip

Starting with two twists associated to invertible functionals $\phi_1$ and $\phi_2$, it is natural to ask whether the composition $\Psi'_{\phi_1} \circ \Psi'_{\phi_2}= F^{-1}(\phi_1 \star \phi_2)$ is associated to an invertible functional. This is the case for the following class of coalgebras. 

\begin{defi}
A coalgebra $C$ is said to be {\uti}  if there exists a coalgebra isomorphism $\chi:C\ra C^{cop}$. 
\end{defi}

If $C$ is \uti, then $C^{cop}\tn D \simeq C\tn D$ as coalgebras (with the tensor co-structures) and we have the induced algebra isomorphism $\left( (C^{cop}\tn D)',\ast \right) \simeq \left( (C\tn D)',\ast \right)$.

\begin{lem}
\label{invstar}
Let $C,D$ be coalgebras, and suppose at least one of them is \uti. Then a functional $\phi:C\tn D\ra\kk$ is invertible with respect to the convolution product $\ast$ if and only if it is invertible with respect to the product $\star$.
\end{lem}
\begin{proof}
If $C$ is \uti, the result follows from the algebra isomorphisms 
$$( (C\tn D)',\star ) \simeq ( (C^{cop}\tn D)', \ast ) \simeq ( (C\tn D)', \ast ) \, .$$
If $D$ is \uti, the result follows from the algebra isomorphisms
$$ ( (C\tn D)',\star ) \simeq ( (C^{cop}\tn D)' , \ast ) \simeq ( (C\tn D^{cop})' , \ast^{op} ) \simeq ( (C\tn D)', \ast^{op} ) \, . $$ 
\end{proof}

Therefore whenever one among $C$ and $D$ is \uti, if we compose two twists which admit counit (i.e. their associated functionals are convolution-invertible) then also the resulting twist admits counit. \\

From Thm. \ref{prop-F} it follows that a twist $\Psi' \in \tw$ is invertible (and hence the associated twisted coproduct $\co_\Psi$ can be `untwisted' with the inverse twist) if and only if $\phi_{\Psi}$ is invertible with respect the $\star$-product. On the other hand  $\co_\Psi$ admits counit if and only if $\phi_{\Psi}$ is invertible with respect the convolution product. In general there may exist twists which do admit counit without being invertible and vice versa, see Example \ref{ex-pairing}. \footnote{Instead if we consider twists which satisfy conditions \eqref{CP1eq}\eqref{CP2eq}, the invertibility of the map $\Psi$ implies the conormality. Indeed if $\Psi$ satisfies \eqref{CP1eq}, then $d^{[\Psi]} \ot c^{[\Psi]}= \varepsilon((\uno{d})^{[\Psi]}) (\due{d})^{[[\Psi]]} \ot (c^{[\Psi]})^{[[\Psi]]}$, where the notation $[[\Psi]]$ indicates a second application of $\Psi$. By applying to it  $(id_C \ot \varepsilon_D ) \Psi^{-1}$ we get the right conormality 
$c \varepsilon_D(d)= c^{[\Psi]} \varepsilon_D(d^{[\Psi]})$. Similarly \eqref{CP2eq} implies the left conormality.
}  \\

We return to the results of universality addressed in Sect. \ref{sec-univ} for generic twists,  now in the case of twists $\Psi'\in\tw$. When the twisted counit exists, i.e. the functional corresponding to $\Psi'$ is invertible, we have a second pair of projection maps $q_C$ and $q_D$ in addition to the projection maps 
$p_C,~p_D$  defined in \eqref{proj}.

\begin{lem}
Let $\Psi'\in\tw$, with $\phi_{\Psi}$ invertible in $[(C\tn D)',\ast]$. Then the maps $q_C:C\tn_{\Psi} D\ra C$ and $q_D:C\tn_{\Psi} D\ra D$ defined as
\begin{equation}\label{proj-q}
q_C(c\tn d) := \tco(\uno{c}\tn d) \, \due{c} \, , \qquad q_D(c\tn d) := \tco(c\tn\due{d}) \, \uno{d} \, . 
\end{equation}
are coalgebra morphisms.
\end{lem}
\begin{proof} 
The map $q_C$ trivially verifies  $\varepsilon_C \circ q_C = \tco$. On  generic elements $c  \in C$, $d \in D$ we have 
\begin{eqnarray*}
(q_C \tn q_C) \circ \co_\Psi (c \tn d)&=& q_C (\uno{c} \tn \uno{d})  \phi_{\Psi}(\due{c}\tn \due{d}) \tco( \tre{c} \tn \tre{d}) \qu{c}  \\
 & = & q_C (\uno{c} \tn \uno{d})  (\phi_{\Psi} \ast \tco) (\due{c}\tn \due{d}) \tre{c}\\
 & = &  q_C (\uno{c} \tn \uno{d}) \varepsilon_\ot (\due{c}\tn \due{d}) \tre{c}\\
  & = & \co_C \circ q_C (c\tn d) 
\end{eqnarray*} 
hence $q_C$ is a coalgebra morphism. The proof for $q_D$ is completely analogous.  
\end{proof}

\begin{lem}
Let $\Psi'\in\tw$, with $\phi_{\Psi}$ invertible in $[(C\tn D)',\ast]$. The maps $\nu:= (p_C\tn q_D)\co_{\Psi}: C\tn_{\Psi} D\ra C\tn D$ and $\sigma:= (q_C\tn p_D)\co_{\Psi}: C\tn_{\Psi} D\ra C\tn D$ are isomorphisms of vector spaces.
\end{lem}
\begin{proof} By definition of the map $\nu$, we have $\nu(c \ot d)= \uno{c} \ot \uno{d} (\phi_\Psi)^{-1}(\due{c} \ot \due{d})$. It is invertible with inverse $\nu^{-1}: c \ot d \mapsto \uno{c} \ot \uno{d} \phi_{\Psi}(\due{c} \ot \due{d})$.  
Analogously, $\sigma(c \ot d)= (\phi_\Psi)^{-1}(\uno{c} \ot \uno{d}) \due{c} \ot \due{d}$ and $\sigma^{-1}: c \ot d \mapsto  \phi_{\Psi}(\uno{c} \ot \uno{d})\due{c} \ot \due{d}$.
\end{proof}

\begin{cor}\label{cor-iso}
Let $\Psi'\in\tw$ such that $\tco=\phi_{\Psi}^{-1}$ is invertible in $[(C\tn D)',\star]$. Then $\widetilde{\Psi}=(p_D\tn p_C)\co_{\Psi}\mu^{-1}$ is a conormal twist, different from $\tau$, and we have the double isomorphisms $C\tn_{\Psi} D\simeq C\tn_{\widetilde{\Psi}} D\simeq C\tn D$.
\end{cor}

\begin{proof} The two different isomorphisms come from Corollary \ref{prtr} applied to the twisted coalgebra $C \ot_\Psi D$ and two different sets of projections. 
The first isomorphism $C\tn_{\Psi} D\simeq C\tn_{\widetilde{\Psi}} D$ is obtained by considering  $p_C$ and $p_D$ as in  \eqref{proj}. For the isomorphism $C\tn_{\Psi} D\simeq C\tn D$ we consider $p_C$ in \eqref{proj} and the projection $q_D$ in \eqref{proj-q}.
It is easy to show that in this case 
$ \widetilde{\Psi}_\nu=(q_D\tn p_C)\co_{\Psi}\nu^{-1}=\tau(p_C\tn q_D)\co_{\Psi} \nu^{-1} =\tau$. (Alternatively, one could consider the pair $q_C$ as in \eqref{proj-q} and $p_D$ in \eqref{proj}, with $\widetilde{\Psi}_\sigma=(p_D\tn q_C)\co_{\Psi}\nu^{-1}=\tau $).
\end{proof}

The previous result shows that all twist maps $\Psi'\in\tw$ which admit compatible counit generate twisted coalgebras which are isomorphic to the untwisted 
one. At the same time we notice that in general $\widetilde{\Psi}\neq\tau$:
\begin{equation*}
\begin{split}
\widetilde{\Psi}(c\tn d) & = (p_D\tn p_C)\co_{\Psi}\mu^{-1}(c\tn d) \\
 & = (p_D\tn p_C)\co_{\Psi}((\phi_{\Psi}^{-1})^{\star}(c_{(2)}\tn d_{(1)}) \, c_{(1)}\tn d_{(2)}) \\
 & = (p_D\tn p_C) \big( (\phi_{\Psi}^{-1})^{\star}(c_{(4)}\tn d_{(1)}) \, \phi_{\Psi}(c_{(2)}\tn d_{(3)}) \, c_{(1)}\tn d_{(2)}\tn c_{(3)}\tn d_{(4)} \big) \\
 & = (\phi_{\Psi}^{-1})^{\star}(c_{(5)}\tn d_{(1)}) \, \phi_{\Psi}(c_{(2)}\tn d_{(4)}) \, \phi^{-1}_{\Psi}(c_{(1)}\tn d_{(2)}) \, \phi^{-1}_{\Psi}(c_{(4)}\tn d_{(5)}) \, d_{(3)}\tn c_{(3)} \, .
\end{split} 
\end{equation*}
Nevertheless, 
these two different conormal twists give rise to isomorphic coalgebras. This leads to the natural problem to study a notion of equivalence for twists solutions of \eqref{COdiagr}.\medskip

\begin{ex}\label{ex-pairing}
As an example, we consider the case of $C$ and $D$ dually paired bialgebras. We take the pairing $\pp{}{}: C \ot D \ra \kk$ as functional. The invertibility of the pairing, which determines the existence of the twisted counit, can be achieved by letting one among $C$ and $D$ to be a Hopf algebra. Let us require that $C$ is a Hopf algebra  (the case $D$ Hopf algebra is completely equivalent). 
The convolution-inverse of $\pp{}{}$ is $\pp{}{}\circ (S_C\tn\id_D)$ that we will denote with $\pp{}{}^{-1}$.
\\
The corresponding  twisted coalgebra structure $(\co_{\pp{}{}},\varepsilon_{\pp{}{}})$ on $C\tn D$ is
\begin{align}
\label{tcop}
\co_{\pp{}{}}(c\tn d) & = \pp{c_{(2)}}{d_{(2)}} \, c_{(1)}\tn d_{(1)} \tn c_{(3)} \tn d_{(3)} \\
\label{tcounp}
\varepsilon_{\pp{}{}} (c\tn d) & = \pp{S(c)}{d} \, 
\end{align} 
whereas the inverse of $\varepsilon_{\pp{}{}}$ with respect to the $\star$ product exists and it is $\varepsilon_{\pp{}{}}^\star (c\tn d)  = \pp{S^2(c)}{d}$ as an easy computation shows.

As we already pointed out, whenever the twisted coproduct admits counit we also have the  twisted structure associated to the convolution-inverse of the functional. This consists of $(C\tn D,\co_{\pp{}{}^{-1}},\varepsilon_{\pp{}{}^{-1}})$ with  
\begin{align}
\label{tcop2}
\co_{\pp{}{}^{-1}}(c\tn d) & = \pp{S(c_{(2)})}{d_{(2)}} \, c_{(1)}\tn d_{(1)} \tn c_{(3)} \tn d_{(3)} \\
\label{tcounp2}
\varepsilon_{\pp{}{}^{-1}} (c\tn d) & = \pp{c}{d} \, 
\end{align}
We remark that in this case the inverse of  $\varepsilon_{\pp{}{}^{-1}}$ with respect to the $\star$ product exists if and only if the antipode $S$ is invertible and in this case it is $\varepsilon_{\pp{}{}^{-1}}^\star (c\tn d)  = \pp{S^{-1}(c)}{d}$.
This is an example of a twist in $\tw$ which admits counit but that in general it is not invertible (see discussion after Lemma \ref{invstar}).
\end{ex}

\subsection{An application}

In this section we consider $C$ to be a bialgebra and $D$ a coalgebra. (What follows can analogously be stated in the case  
$C$ coalgebra and $D$ bialgebra.)
\\
Given any functional on $C \ot D$  satisfying $\phi(c'c\tn d)=\phi(c',d_{(1)})\phi(c,d_{(2)})$ for any $c,c' \in C,~d\in D$, the map $\lambda_\phi : C \ot D \ra D$ defined as $\lambda_\phi(c\tn d) := d_{(1)}\phi(c \ot d_{(2)})$ is a left action of $C$ on $D$.
For example, for $C$ and $D$ dually paired bialgebras, the pairing satisfies the above condition and the associated map $\lambda_{\pp{}{}}$ is the left coregular action of $C$ on $D$.

One notable property of the twisted co-structures  associated to a functional $\phi$ consists in the fact that they restore some kind of compatibility between the $C$ action $\lambda_\phi$ and the coalgebra structure of $D$.
The map $\lambda_\phi$ does \textit{not} make $D$ a left $C$-module coalgebra, unless $\phi=\varepsilon_\ot$. Indeed $D$ is a coalgebra in the category of left $C$-modules ${_C\textit{M}}$ if the coproduct $\co_D$ and the counit $\varepsilon_D$ are morphisms in ${_C\textit{M}}$, and it is easy to check that this is not the case for $\phi \neq \varepsilon_\ot$. An equivalent way to say that $D$ is a left $C$-module coalgebra  consists in requiring that $\lambda_\phi$ is a coalgebra map for $C\tn D$  endowed with the tensor product co-structures $\co_{\tn}$ and $\varepsilon_{\tn}$. This amounts to the commutativity of the following diagrams
\begin{equation}
\label{d1}
\xymatrix @C4pc{
C \ot D \ar[r]^{\lambda_\phi} \ar[d]_{\co_{\ot}} & D  \ar[d]^{\co_D}\\
C\tn D\tn C\tn D \ar[r]^-{\lambda_\phi \ot\lambda_\phi} & D\tn D
}
\qquad\qquad
\xymatrix{
C\ot D \ar[r]^-{\lambda_\phi} \ar[dr]_{\varepsilon_\ot} & D \ar[d]^{\varepsilon_D}  \\
& \kk }
\end{equation}
which fail for $\phi \neq \varepsilon_\ot$. However, as the next Proposition shows, we can make the above diagrams commutative by twisting in a suitable way the tensor co-structures $\co_{\tn}$ and $\varepsilon_{\tn}$. In this manner, after considering the twisted coalgebra co-structures on $C\tn D$, the map $\lambda_\phi$ becomes a coalgebra map. 

\begin{prop}
\label{diagrprop}
Let $C$ be a bialgebra and $D$ a coalgebra. Let $\lambda_\phi:C\tn D\ra D$ defined as $\lambda_\phi(c\tn d) := d_{(1)}\phi(c \ot d_{(2)})$ be a left action of $C$ on $D$, with $\phi$ a convolution invertible  functional on $C\tn D$. Consider the twisted coalgebra $(C\tn_{\phi^{-1}} D, \co_{\phi^{-1}}, \varepsilon_{\phi^{-1}})$. 
Then the following diagrams commute:
\begin{equation}
\label{d2}
\xymatrix @C4pc{
C \ot D \ar[r]^{\lambda_\phi} \ar[d]_{\co_{\phi^{-1}}} & D  \ar[d]^{\co_D}\\
C\tn D\tn C\tn D \ar[r]^-{\lambda_\phi\ot\lambda_\phi} & D\tn D
}
\qquad\qquad
\xymatrix{
C\ot D \ar[r]^-{\lambda_\phi} \ar[dr]_{\varepsilon_{\phi^{-1}}} & D \ar[d]^{\varepsilon_D}  \\
& \kk }
\end{equation}
so that $\lambda_\phi$ is a coalgebra map between $(C\tn_{\phi^{-1}} D, \co_{\phi^{-1}}, \varepsilon_{\phi^{-1}})$ and $(D,\co_D,\varepsilon_D)$.
\end{prop}
\begin{proof}
Let us verify the commutativity of the  first diagram:
\begin{equation*}
\begin{split}
(\lambda_\phi\tn\lambda_\phi)\co_{\phi^{-1}}(c\tn d) & = \lambda_\phi (\uno{c} \ot \uno{d})\phi^{-1}(\due{c} \ot \due{d}) \lambda_\phi(\tre{c} \tn \tre{d}) \\
 & =\uno{d} \phi(\uno{c} \ot \due{d}) \phi^{-1}(\due{c} \ot \tre{d}) \lambda_\phi(\tre{c} \ot \qu{d}) \\
 & = \uno{d} \lambda_\phi(c \ot \due{d}) = \co_D(\lambda_\phi(c\tn d)) \; .
\end{split}
\end{equation*} 
The commutativity of the diagram involving the counit is trivial.
\end{proof}  \bigskip

\subsection{How to generate new twists from existing ones}\label{se-nt}

We have seen that the twists associated to functionals generate twisted coalgebras which eventually turn out to be isomorphic to 
the untwisted ones (see Corollary \ref{cor-iso}).  Nevertheless, they can be composed with solutions of \eqref{COeq} to produce new twists. 
\begin{prop}\label{prop-nt}
Consider two twist maps $\chi$ and $\Psi$. Suppose that $\chi$ is a solution of \eqref{COeq} and $\Psi'\in\tw$. Then the twist given by the composition $\nt:=\Psi^\prime \circ\chi$ is a solution of \eqref{COeq}.
\end{prop}
\begin{proof}
We compute both sides of eq. \eqref{COeq} for $\nt$ by using the eqs. \eqref{LC1} and \eqref{LC2} to move $\Psi$ to the left.  The left hand-side is 
\begin{equation*}
\begin{split}
(\id_D & \tn\id_C \tn \Psi'\circ\chi)(\id_D\tn\co_C\tn\id_D)(\Psi' \circ\chi\tn\id_D)(\id_C\tn\co_D) = \\
  & = (\id_D\tn\id_C \tn \Psi ' \circ\chi)(\Psi'\tn\id_C\tn\id_D)(\id_D \ot \co_C\tn\id_D)(\chi \tn\id_D)(\id_C\tn\co_D) = \\
  & = (\Psi' \ot \Psi') \left(\mbox{l.h.s of \eqref{COeq} for } \chi \right) \; .
\end{split}
\end{equation*}
Similarly the right hand-side is
\begin{equation*}
\begin{split}
(\Psi'\circ & \chi \tn\id_D\tn\id_C)(\id_C\tn\co_D\tn\id_C)(\id_C\tn\Psi'\circ\chi)(\co_C\tn\id_D) = \\
  & =(\Psi'\circ  \chi \tn\id_D\tn\id_C)(\id_C\tn\id_D\tn\Psi')(\id_C\tn\co_D \ot \id_C) (\id_C \ot \chi)(\co_C\tn\id_D)\\
  & = (\Psi ' \tn\Psi')\left(\mbox{r.h.s of \eqref{COeq} for } \chi \right) \; .
\end{split}
\end{equation*}
\end{proof}

This result shows that the left composition by a $\Psi\in\tw$ preserves the coassociativity. On the other side, we remark that the possible additional properties enjoyed by $\Psi$ and $\chi$ are in general lost. For instance $\nt \not\in \tw$ for a generic  $\chi \not\in \tw$. Furthermore if $\chi$ satisfies \eqref{CP1eq} and \eqref{CP2eq}, the composition $\nt$ in general is no longer a solution of these equations. In general if $\chi$ is conormal, the new twist $\nt$ will not be conormal anymore (see Ex. \ref{ex-nt}) as can be seen with some algebra. 
Thus, provided the existence of a counit for $\nt$, the Prop. \ref{prop-nt} above gives in particular a systematic way to obtain $Z$-conormal coassociative twists out of conormal ones.

At the present time, we lack of a general criterion for the existence of a compatible counit for $\nt$, thus its expression. Another interesting point to be investigated is the existence of  projections to $C$ and $D$, the study of the invertibility  of their associated  map $\mu$ (in view of Corollary \ref{prtr}) and the possible isomorphism between the twisted coalgebra 
$C \ot_{\nt} D$  and the starting one $C \ot_\Psi D$. \\

\begin{ex}\label{ex-nt}
As an example of a new non-conormal twist $\nt$ obtained from a conormal twist $\chi$ and  which admits a compatible counit
we consider the following.  Let $H$ be a bialgebra. Let $C$ be a right $H$-module coalgebra with action $\triangleleft: C \ot H \ra C$, $c \ot h \mapsto c \triangleleft h$. Let $D$ be a right $H$-comodule coalgebra with coaction $\rho: D \ra D \ot H$, $d \mapsto \suze{d} \ot \suno{d}.$ 
The twist $\chi: c \ot d \mapsto \suze{d} \ot c \rac \suno{d}$ is conormal and satisfies eqs. \eqref{CP1eq}, \eqref{CP2eq} (cf. \cite[Ex. 3.2]{cae}). The resulting twisted coalgebra $C \ot_\chi D$ is Molnar's smash coproduct.\\
Let $\phi$ be a convolution-invertible functional on $C \ot D$ such that $\phi(c \rac h \ot d)= \varepsilon_H(h) \phi(c \ot d)$ for all $c \in C, ~d \in D, ~h \in H$. 
Then the composed  twist $\chi_{_{\phi}}:=\chi_{_{\Psi_\phi}}= \Psi'_\phi \circ \chi$ generates a coassociative coalgebra with counit  
$\varepsilon_{\chi_{_{\phi}}} = \phi^{-1}$.

To prove this last statement we  use the following identities following from the hypothesis: for all $c \in C, ~d \in D, ~h \in H$
\begin{eqnarray*}
\uno{(c \rac h)} \ot \due{(c \rac h)} = \uno{c} \rac \uno{h} \ot \due{c} \rac \due{h} \quad &;& 
\varepsilon_C(c \rac h)= \varepsilon_C(c) \varepsilon_H(h) ~ ;
\\
\uno{(\suze{d})} \ot \due{(\suze{d})} \ot \suno{d} = \suze{(\uno{d})} \ot \suze{(\due{d})} \ot \suno{(\uno{d})} 
\suno{(\due{d})} &;&
\varepsilon_D(\suze{d}) \suno{d}= \varepsilon_D(d) 1_H \; .
\end{eqnarray*}  
Firstly, we compute the twisted coproduct:
\begin{eqnarray*}
\co_{\chi_{_{\phi}}} (c \ot d) &=& \uno{c} \ot \Psi'_\phi \left( \suze{(\uno{d})}\ot \due{c} \rac \suno{(\uno{d})} \right) \ot \due{d}
\\
&=& \uno{c} \ot \phi \left( \uno{(\due{c} \rac \suno{(\uno{d})})} \ot \due{(\suze{(\uno{d})})} \right) \uno{(\suze{(\uno{d})})} \ot \due{(\due{c} \rac \suno{(\uno{d})})} \ot \due{d}=
\\
&=& \uno{c} \ot \phi \left( \due{c} \rac \uno{(\suno{(\uno{d})})} \ot \due{(\suze{(\uno{d})})} \right) \uno{(\suze{(\uno{d})})} \ot \tre{c} \rac\due{ (\suno{(\uno{d})})} \ot \due{d}=
\\
&=&   \uno{c} \ot \phi \left( \due{c} \rac \uno{\left(\suno{(\uno{d})}\suno{(\due{d})}\right)} \ot {\suze{(\due{d})}} \right) {\suze{(\uno{d})}} \ot \tre{c} \rac\due{( \suno{(\uno{d})}\suno{(\due{d})})} \ot \tre{d}=
\\
&=&  \uno{c} \ot \phi \left( \due{c} \ot \suze{(\due{d})} \right) \suze{(\uno{d})} \ot \tre{c} \rac  \left( \suno{(\uno{d})}\suno{(\due{d})} \right) \ot \tre{d}
\end{eqnarray*}
Next we show that $(\phi^{-1} \ot \id) \co_{\chi_{_{\phi}}} (c \ot d) = c \ot d ~$:
\begin{eqnarray*}
(\phi^{-1} \ot \id) \co_{\chi_{_{\phi}}} (c \ot d)  &=& \phi^{-1} \left(\uno{c} \ot \suze{(\uno{d})} \right)  \phi \left( \due{c} \ot \suze{(\due{d})} \right)  \tre{c} \rac  \left(\suno{(\uno{d})}\suno{(\due{d})} \right) \ot \tre{d} 
\\
&=& \phi^{-1} \left(\uno{c} \ot \uno{(\suze{(\uno{d})})} \right)  \phi \left( \due{c} \ot \due{(\suze{(\uno{d})})} \right)  \tre{c} \rac  \suno{(\uno{d})} \ot \due{d} \\
&=& 
c  \rac \suno{(\uno{d})} \varepsilon_H(\suze{(\uno{d})}) \ot \due{d} = c \rac \varepsilon_D(\uno{d}) \ot \due{d} = c \ot d 
\end{eqnarray*}
and finally that  $(\id \ot \phi^{-1}) \co_{\chi_{_{\phi}}} (c \ot d) = c \ot d ~$:
\begin{eqnarray*}
(\id \ot \phi^{-1}) \co_{\chi_{_{\phi}}} (c \ot d) &=& \uno{c} \ot \phi \left( \due{c} \ot \suze{(\due{d})} \right) \suze{(\uno{d})} \phi^{-1} \left( \tre{c} \rac  (\suno{(\uno{d})}\suno{(\due{d})}) \ot \tre{d} \right) 
\\
&=& \uno{c} \ot \phi \left( \due{c} \ot \due{(\suze{(\uno{d})})} \right) \uno{(\suze{(\uno{d})})} \phi^{-1} \left( \tre{c} \rac  \suno{(\uno{d})} \ot \due{d} \right) 
\\
&=& \uno{c} \ot \phi \left( \due{c} \ot \due{d} \right) \uno{d} \phi^{-1} \left( \tre{c} \ot \tre{d} \right) = c \ot d
\end{eqnarray*}
where we have also  made use of the property  $\phi^{-1}(c \rac h \ot d)= \varepsilon_H(h) \phi^{-1}(c \ot d)$ of $\phi^{-1}$,  inherited from the analogous property of $\phi$. 
\end{ex}

\section{On the equivalence of twists}\label{se-equiv}

\begin{defi}\label{def-eq}
Two twists $\Psi_1, ~\Psi_2$ are said to be equivalent if there exists a coalgebra isomorphism $\theta :C \ot D \ra C \ot D$  such that
\be\label{equiv}
(\id_C \ot \Psi_2^\prime \ot \id_D)(\theta \ot \theta) = (\theta \ot \theta) (\id_C \ot \Psi_1^\prime \ot \id_D)  .
\ee
\end{defi}

\noindent The above condition is for instance satisfied by any coalgebra isomorphism which factorizes as $\theta= \alpha \ot \beta$ with $\alpha:C \ra C$ and $\beta:D \ra D$. 
\begin{lem}\label{lem-inveq}
Suppose $\Psi_1, ~\Psi_2$ are two equivalent twists as in Def. \ref{def-eq}. Then $\Psi_1$ is $Z_1$-conormal if and only if $\Psi_2$ is $Z_2$-conormal, with $Z_2 = \hat{\theta}  ~Z_1~  \theta^{-1}$, where $\hat{\theta}:= \tau ~ \theta ~ \tau$. 
\end{lem}
\begin{proof}
Let us suppose that $\Psi_1$ is $Z_1$-conormal. Then
\begin{eqnarray*}
(\varepsilon_\tn^\tau ~ Z_2 \ot \id_\ot ) \Delta_{\Psi_2} &=& (\varepsilon_\ot^\tau \ot \id_\ot)(\hat{\theta} Z_1 
\theta^{-1} \ot \id_\ot)(\id_C \ot \Psi^\prime_2 \ot \id_D) \Delta_\ot \\
&=&  (\varepsilon_\ot^\tau \ot \id_\ot)(\hat{\theta} ~ Z_1~
\theta^{-1} \ot \id_\ot)(\theta \ot \theta)(\id_C \ot \Psi^\prime_1 \ot \id_D) \Delta_\ot \theta^{-1} \\
&=&  (\varepsilon_\ot^\tau \ot \id_\ot)(\hat{\theta} ~ Z_1 \ot \theta) \Delta_{\Psi_1} \theta^{-1} \\
&=& \theta (\varepsilon_\ot^\tau Z_1 \ot \id_\ot )  \Delta_{\Psi_1} \theta^{-1} = \id_\ot
\end{eqnarray*}
and similarly $(\id_\ot \ot \varepsilon_\tn^\tau ~ Z_2) \Delta_{\Psi_2}= \id_\ot$. The opposite implication is analogous.
\end{proof}
\noindent We observe that hence in particular a conormal twist can only be equivalent to another conormal twist. 
Equivalent twists generate isomorphic twisted coalgebras: 
\begin{prop}\label{prop-eq}
Let $\Psi_1, ~\Psi_2$ be two equivalent twists as in Def. \ref{def-eq} (which satisfy \eqref{COeq}).  
Then the map $\theta: (C \ot_{\Psi_1} D,\co_{\Psi_1})  \ra (C \ot_{\Psi_2} D, \co_{\Psi_{2}})$ intertwines the  (coassociative) coproducts $\co_{\Psi_1}$ and  $\co_{\Psi_{2}}$. In particular if  $\co_{\Psi_1}, ~\co_{\Psi_2}$ admit counits, then the map $\theta$ is an isomorphism of coalgebras.  
\end{prop}
\begin{proof}
From \eqref{equiv}, it is straightforward that $\theta$ satisfies 
$(\theta \ot \theta) \co_{\Psi_1}= \co_{\Psi_2} \theta$.  From Lemma \ref{lem-inveq}  it follows easily that 
$\varepsilon_{\Psi_1} = \varepsilon_{\Psi_2} \theta$. 
\end{proof}

We notice that this notion of equivalence of twists is not exhaustive to capture when two twisted coalgebras are isomorphic. For instance in Corollary \ref{cor-iso} we have two coalgebra isomorphisms $C \ot_\Psi D \simeq C \ot D$ and $C \ot_{\Psi} D \simeq C \ot_{\widetilde\Psi} D$, but $\Psi$ is not conormal while $\widetilde\Psi$ does.   Furthermore  $C \ot D \simeq C \ot_{\widetilde\Psi} D$ but from 
Def. \ref{def-eq}, the equivalence class of the trivial twist consists of the sole $\tau$. 

Equivalent twists characterize a finer notion of isomorphism of twisted coalgebras. 
\begin{defi}\label{str-eq}
Two twisted coalgebras  $(C \ot_{\Psi_1} D,\co_{\Psi_1}),~ (C \ot_{\Psi_2} D, \co_{\Psi_{2}})$ are said to be
strongly isomorphic if there exists a coalgebra isomorphism  $\theta: (C \ot_{\Psi_1} D,\co_{\Psi_1})  \ra (C \ot_{\Psi_2} D, \co_{\Psi_{2}})$  which in addition intertwines the twist maps in the sense that
$$
(\id_C \ot \Psi_2^\prime \ot \id_D)(\theta \ot \theta) = (\theta \ot \theta) (\id_C \ot \Psi_1^\prime \ot \id_D) \; .$$
\end{defi}
\noindent
From Prop. \ref{prop-eq} we immediately get
\begin{prop}
Two twisted coalgebras  $(C \ot_{\Psi_1} D,\co_{\Psi_1}),~ (C \ot_{\Psi_2} D, \co_{\Psi_{2}})$ are 
strongly isomorphic if and only if the twists $\Psi_1 ,~ \Psi_2$ are equivalent.
\end{prop}

\section{The dual case: twisting tensor algebra structures}\label{se-dual}

\textit{Notation.} 
Given a unital $\kk$-algebra $A$, we denote by $m_A$ and $\eta_A$ the multiplication and the unit map respectively. 
We also use the standard notation $1_A=\eta_A(1_\kk)$ for the unit element of $A$. We denote with
$A^{op}$ the opposite algebra $(A, m_A \circ \tau, \eta_A)$.  Given another algebra $B$, when we refer to $A \ot B$ as an algebra, without further specifications,  we always mean $(A \ot B, m_\ot, \eta_\ot)$ with the standard tensor product structures $m_\ot=(m_A \ot m_B) \circ (id \ot \tau \ot \id)$, $\eta_\ot= \eta_A \ot \eta_B$.\\

In this section we dualize the results obtained in the study of twisted tensor co-structures to the dual case of the tensor product 
of two algebras. We will omit to write  most of the proofs.  
We consider two (fixed)  unital and associative $\kk$-algebras $A$ and $B$. We refer to a  $\kk$-linear map
 $\ta:B\tn A\ra A\tn B$, $\ta(b \ot a):= \taa{a} \ot \taa{b}$ to as a `twist map' and use the same terminology 
for the map $\ta':=\tau \circ \ta$.   
We recall some relevant results from \cite[\S 2]{cae}.

\begin{thm}
Let $\ta$ be a twist map. The map  $m_\ta:(A\tn B)\tn (A\tn B)\ra A\tn B$  given by  
\be \label{opsi}
m_\ta:=(m_A\tn m_B)(\id_A\tn \ta \tn \id_B)
\ee
defines an associative product on the vector space $A \ot B$ if and only if the twist map $\ta$ satisfies
\begin{multline}
\label{Oeq}
(\id_A\tn m_B)(\ta\tn\id_B)(\id_B\tn m_A\tn\id_B)(\id_B\tn\id_A\tn \ta) = \\ (m_A\tn\id_B)(\id_A\tn \ta)(\id_A\tn m_B\tn\id_A)(\ta\tn\id_B\tn\id_A) \, .
\end{multline}
\end{thm}

We denote with $A \ot_\ta B$ the vector space $A\ot B$ equipped with the twisted product $m_\ta$.
It is possible to make $A\ot_\ta B$ into a unital algebra provided  that $\ta$ satisfies some further conditions. The most common choice is the following:

\begin{defi}\label{def-un}
A twist map $\ta$ is said to be right (resp. left) normal if for all $a \in A$, $b \in B$
\begin{equation}\label{nor}
\ta  (1_B\tn a)= a \tn 1_B \, , \quad  (\mbox{resp.  } \ta (b \tn 1_A)=1_A\tn b)\; .
\end{equation}
It is said to be normal when it is both right and left normal.
\end{defi}

\begin{lem} 
Let $\ta$ be a twist map. The tensor unit $\eta_\ot:=\eta_A \ot \eta_B$ is compatible with the twisted product $m_{\ta}$, i.e. $m_\ta \circ (\id \ot \eta_\ot) = \id =m_\ta \circ(\eta_\ot \ot \id)$  if and only if $\ta$ is normal.
\end{lem}

The associativity of $m_{\ta}$ and the property of $\eta_{\tn}$ to be a compatible unit are completely independent and clearly   there are examples of associative algebras $(A \ot B, m_{\ta})$ which do not admit $\eta_{\ot}$ as unit (or even more in general which do not admit unit at all) and there are normal twists $\ta$ which do not satisfy the associativity condition \eqref{Oeq}. \\

Under the assumption of normality of $\ta$, the condition  \eqref{Oeq} can be split in the following two conditions
\begin{equation}
\label{CP1a}
(\id_A\tn m_B)(\ta\tn\id_B)(\id_B\tn \ta) =\ta (m_B \ot \id_A )
\end{equation}
and
\begin{equation}
\label{CP2a}
 (m_A\tn \id_B)(\id_A\tn\ta)(\ta \tn \id_A) =\ta (\id_B \ot m_A )\, .
\end{equation}

\begin{thm}
\label{thm-cae-alg} 
Let $A,~B$ and $X$ be unital $\kk$-algebras. The following two conditions are equivalent:
\begin{enumerate}
\item
There exists an algebra isomorphism $X \simeq A \ot_\ta B$ for some normal twist map $\ta$ solution of \eqref{Oeq};
\item X factorizes through $A$ and $B$, i.e. there exist algebra morphisms $u_A: A \ra X$ and $u_B : B \ra X $ such that the map $\xi:= m_X \circ (u_A \ot u_B): A\ot B \ra X$ is an isomorphism of vector spaces.
\end{enumerate}
\end{thm}
The proof is based on the fact that if $A \ot_\ta B$ is an algebra associated to a normal twist $\ta$, then the maps 
\begin{equation}
\label{projpi_alg}
\begin{split}
\pi_A & : A \ra A \ot_\ta B \, , \quad a \mapsto a \ot 1_B \\ 
\pi_B & : B \ra A \ot_\ta B \, , \quad b \mapsto 1_A \ot b
\end{split}
\end{equation}    
are algebra morphisms, and the map 
\begin{equation}
\label{xi}
\xi:=m_\ta \circ (\pi_A \ot \pi_B):A \ot B   \ra A \ot_\ta B
\end{equation} 
is an isomorphism of vector spaces, being  $\xi(a \ot b)= a \ot b$. For the opposite implication, one constructs the normal twist map as $\ta=\xi^{-1}\circ m_X \circ  (u_B \ot u_A)$.\\

\noindent
Twisted algebras $A \ot_\ta B$ are characterized by the following universal property:   

\begin{prop}  
\label{univ-cae-alg}
Let $A$ and $B$ be algebras, $\ta$ a normal twist  satisfying \eqref{Oeq}. Let 
$X$ be a unital algebra and $u_A: A \ra X, ~ u_B:B\ra X$ algebra morphisms such that 
\be
m_X \circ (u_B \ot u_A) = m_X  \circ (u_A \ot u_B) \circ \ta \; .
\ee
Then there exists a unique algebra morphism $\omega:A\ot_\ta B \ra X$ such that the following diagram commutes
\begin{equation}
\label{univ-diagr}
\xymatrix{
  & A\ot_{\ta}B  \ar@{.>}[dd]^{\omega} &   \\
A \ar[ur]^{\pi_A}\ar[dr]_{u_A}&       & B\ar[ul]_{\pi_B}\ar[dl]^{u_B} \\
  & X &
}
\end{equation}
\end{prop}


\subsection{Non-normal twists}\label{sec-nt}

 We address to twist maps $\ta$ which are not normal but still solutions of the associativity constraint \eqref{Oeq}. We   generalize the notion of normality into the following

\begin{defi} 
A twist map $\ta$ is said to be right $z$-normal if there exists an element  $z=z_A \ot z_B \in A \ot B$ (possible sum  understood) such that 
\begin{equation}
\label{rZn}
 z_A \taa{a} \ot \taa{(z_B)}= a \ot 1_B\, , \quad \forall a \in A
\end{equation}
and it is said to be left $z$-normal if 
\begin{equation}
\label{lZn}
 \taa{( z_A)}  \ot  \taa{b} z_B = 1_A \ot b\, , \quad \forall b \in B \; .
\end{equation}
It is said to be $Z$-normal if it is both left and right $z$-normal.
\end{defi}

\begin{lem}
Let $\ta$ be a twist map.
\begin{enumerate}[(i)] 
\item  $\ta$ is left (resp.right) normal if and only if $\ta$ is left (resp. right) $1_A \ot 1_B$-normal.
\item If $\ta$ is left (resp. right) normal, then $\ta$ is left (resp. right) $\ta(1_B \ot 1_A)$-normal.
\item If  $\ta$ satisfies conditions \eqref{CP1a}\eqref{CP2a}, then  $\ta$ is left (resp. right) $\ta(1_B \ot 1_A)$-normal if and only if 
$\ta$ is left (resp. right) normal.
\end{enumerate}
\end{lem}

\begin{prop}
\label{prop-unita}
A twist map $\ta$ which satisfies \eqref{Oeq} is $z$-normal for some $z=z_A \ot z_B \in A \ot B$ if and only if the $\kk$-linear map
\be
\eta_{z} = \kk \ra A \ot B \, , ~~1 \mapsto z_A \ot z_B  
\ee 
is a compatible unit for $A \ot_{\ta}B$.
\end{prop}
In the following, we will rather use the notation $\tu$ to indicate the unit $\eta_z$ of the algebra $A \ot_\ta B$ associated to a $z$-normal twist $\ta$. 


\subsection{Universal properties and factorization}\label{sec-univ-alg}

In this subsection we consider a generic twist map $\ta$ such that the (non necessarily associative) algebra  $(A \ot_\ta B, m_{\ta})$ is unital, with unit $\tu$, $\tu(1)=z_A \ot z_B$. Further properties, for example  the associativity condition \eqref{Oeq}, are explicitly required only when necessary.

\begin{lem} 
The inclusion maps $i_A: A \ra A\tn_{\ta} B \, , ~ i_B: B \ra A \tn_{\ta}B$ defined respectively by
\begin{equation}
\label{proj-alg}
i_A (a) := a ~ z_A \ot z_B \, , \quad i_B (b) :=z_A \ot z_B~b \, ,  \quad \forall a \in A,~ b \in B  .
\end{equation} are algebra morphisms.
\end{lem}
\begin{proof} \textit{(sketch).}
Using the notiation $\cdot$ to indicate the untwisted tensor product multiplication and $\cdot_\ta$ for the twisted multiplication, 
one can observe that for all $a,a' \in A$ and $b,b' \in B$ it is 
$$
[(a \ot 1_B)\cdot(z_A \ot z_B)] \cdot_\ta (a' \ot b)= (a \ot 1_B)\cdot[(z_A \ot z_B) \cdot_\ta (a' \ot b)]
$$
and 
$$
(a \ot b)\cdot_\ta [(z_A \ot z_B) \cdot (1 \ot b')]= [(a \ot b)\cdot_\ta (z_A \ot z_B)] \cdot (1 \ot b') \; 
$$
and hence prove that $i_A$ and $i_B$ are algebra morphisms.
\end{proof}
\noindent
 Clearly, these maps reduce to  the inclusions $\pi_A,\pi_B$ in \eqref{projpi_alg} when the twist $\ta$ is normal.

\begin{prop}
\label{invmu-alg} 
Let $\ta$ be a twist map and $z_a \ot z_B$  the unit element of $A \ot_\ta B$.  Let $i_A,~i_B$ be the algebra maps defined in \eqref{proj-alg}. The  map $\mu:A \ot B \ra A\tn_{\ta} B$, 
$$ \mu:=m_\ta (i_A \ot i_B): \quad  a \ot b \mapsto a z_A \ot z_B b$$ 
is an isomorphism of vector spaces if and only if the unit element  $z_A \ot z_B$ is invertible in the algebra $A \ot B^{op}$. 
\end{prop}

\begin{proof}\textit{(sketch).} Let us first assume that $z_A \ot z_B$ is invertible in $A \ot B^{op}$ and denote its inverse by $z_A^{\star} \ot z_B^{\star}$, so that $z_A z_A^\star \ot z_B^\star z_B = 1_A \ot 1_B = z_A^\star z_A \ot z_B  z_B^\star.$
Then 
$$ \mu^{-1}(a\tn b) = a z_A^{\star} \ot z_B^\star b$$ is the inverse of $\mu$. 
Conversely, if $\mu$ is invertible, then $z_A^{\star} \ot z_B^{\star}:= \mu^{-1}(1_A \ot 1_B)$ is the inverse of the unit element $z_A \ot z_B$ in $A \ot B^{op}$.  
\end{proof}
\noindent
Note that the above result holds for a generic twist, not necessarily solutions of $\eqref{Oeq}$.
\\

\noindent As a consequence of Thm. \ref{thm-cae-alg} we have the following
\begin{cor}\label{prtr-alg}
Let $(A \ot_\ta B, m_\ta, \tu)$ be a twisted algebra associated to a twist map $\ta$ which satisfies the associativity condition \eqref{Oeq}. 
Suppose there exist two algebra morphisms $i_A: A \ra A \ot_\ta B$, $i_B: B \ra A \ot_\ta B$ such that the map 
$\mu= m_\ta (i_A\tn i_B)$ is invertible. Then $\widetilde{\ta}:=\mu^{-1} m_\ta (i_B\tn i_A)$ is a normal twist, and $\mu:A\tn_{\ta}B\ra A\tn_{\widetilde{\ta}}B$ realizes an algebra isomorphism.
\end{cor}

We conclude with the following universal characterization of the twisted algebra $A\ot_\ta B$ (cf. Prop. \ref{univ-cae-alg}):

\begin{thm}
Let $A\tn_{\ta} B=(A\tn B,m_{\ta},\tu)$ be a twisted algebra for some twist $\ta$ solution of $\eqref{Oeq}$, and $i_A,~i_B$ the inclusions introduced in \eqref{proj-alg}. Let $X$ be an associative algebra, together with algebra morphisms $j_A: A \ra X \, , \, j_B:B\ra X$ such that
\begin{equation}
\label{twcond-alg}
m_X \circ ( \, j_B \ot j_A\, ) = m_X  \circ ( \, j_A \ot j_B \, ) \circ \ta \, , \qquad 
m_X \circ ( \, j_A \ot j_B\, )\circ  \tu=\eta_X\, \, .
\end{equation}
Then there exists an algebra morphism $\omega : A\ot_{\ta}B \ra X$ such that $j_A=\omega \circ i_A$ and $j_B=\omega \circ i_B$, i.e. the following diagram commutes
\begin{equation}
\label{cdu-alg}
\xymatrix{
  & A\ot_{\ta}B  \ar@{.>}[dd]^{\omega} &   \\
A \ar[ur]^{i_A}\ar[dr]_{j_A}&     & B\ar[ul]_{i_B}\ar[dl]^{j_B} \\
  & X &
}
\end{equation}
\end{thm}

\begin{proof} \textit{(sketch).}
Set $\omega:= m_X \circ  ( \, j_A\ot j_B \, )$.
\end{proof}

\begin{rem} 
In the above Theorem the algebra map $\omega$ is unique for twists $\ta$ such that the corresponding map $\mu=m_\ta \circ(i_A \ot i_B)$ is surjective.
\end{rem}

\begin{rem}
Suppose $X$  itself is a unital twisted algebra $X=A\otimes_{\chi} B$ associated to some twist $\chi$ solution of \eqref{Oeq}, and the map $\mu_{\chi}=m_\chi \circ(i_A \ot i_B)$ is invertible (i.e. the case of Corollary \ref{prtr-alg}).  Then condition \eqref{twcond-alg} is satisfied by 
the normal twist $\ta= \widetilde{\chi} =\mu_{\chi}^{-1} \circ m_\chi (i_B\otimes i_A)$ and the algebra isomorphism $\omega:A\otimes_{\chi}B\ra A\otimes_{\widetilde{\chi}} B$ is the same of Corollary \ref{prtr-alg}.
\end{rem}

\subsection{Twists from morphisms}\label{se-functional-alg}

In analogy with what done in Sect. \ref{se-functional}, we  now describe a class of solutions of condition \eqref{Oeq} which are intrinsically non normal. Conditions \eqref{CP1a}\eqref{CP2a} are no longer necessary for the associativity of $m_{\ta}$ once we remove the condition of normality. \\

\noindent
We denote with $\m$ the category of left $B$-modules and right $A$-modules.
We have that $B\tn A\in Obj(\m)$ via the left $B$-action $_B\rho$  and the right $A$-action $\rho_A$ defined on generic $a,a'\in A$ and $b,b'\in B$ as 
$$ _B\rho(b'\tn b\tn a)=b'b\tn a \, \qquad \rho_A(b\tn a\tn a')=b\tn aa' \, .$$
Similarly, we have that $A\tn B\in Obj(\m)$ via the left $B$-action $_B\rho^{\tau}$  and the right $A$-action $\rho_A^{\tau}$ defined as 
$$ _B\rho^{\tau}(b'\tn a\tn b)=a\tn b'b \, \qquad \rho_A^{\tau}(a\tn b\tn a')=aa'\tn b \, .$$

\begin{thm}
Any twist map $\ta$ which is a morphism in $\m$ satisfies condition \eqref{Oeq} and hence generates an associative product $m_\psi$ on $A\tn B$.
\end{thm}
\begin{proof} \textit{(sketch).}
We use $R\in Mor(\m)$ to move $R$ towards right in both sides of \eqref{Oeq}, so that it becomes
$$ (m_A\tn\id_B)(\id_A\tn {_B}\rho^{\tau})(R\tn R) = (\id_A\tn m_B)(\rho_A^{\tau}\tn\id_B)(R\tn R) \, .$$
It is a direct computation to check that $(m_A\tn\id_B)(\id_A\tn {_B}\rho^{\tau})=(\id_A\tn m_B)(\rho_A^{\tau}\tn\id_B)$ on $A\tn B\tn A\tn B$.
\end{proof}

We observe that $\ta\in Mor(\m)$ if and only if the map $\ta'=\tau \circ\ta:B\tn A\ra B\tn A$ does. The condition $\ta '\in Mor(\m)$
can be expressed as the validity  of the conditions 
\begin{align}
\label{LC1a}
\ta^{\prime}(m_B\tn\id_A) & =(m_B\tn\id_A) (\id_B\tn\ta^{\prime}) \\
\label{LC2a}
\ta^{\prime}(\id_B\tn m_A) & = (\id_B\tn m_A)(\ta^{\prime}\tn\id_A) \,
\end{align} 
\medskip
This allows to conclude  that
\begin{cor}\label{cor-comp-alg}
The space of solutions of \eqref{LC1a} and \eqref{LC2a} is a unital algebra with multiplication given by the composition of morphisms.
\end{cor}
\noindent 
Furthermore, the above conditions   \eqref{LC1a} and \eqref{LC2a} are equivalent to the following ones
\begin{align}
\label{eLC1a}
\taa{b} \ot \taa{a}& = b \taa{(1_B)} \ot \taa{a} \\
\label{eLC2a}
\taa{b} \ot \taa{a}& = \taa{b} \ot \taa{(1_A)} {a}  \,  
\end{align}
for all $ a \in A,~b \in B$ and these yield a simplified expression of the twisted product:
\begin{equation}
\label{cot-alg}
m_{\ta}(a\tn b \ot a' \ot b')  = aa' \taa{1_A} \ot \taa{1_B} b b' \; , \quad \forall a,a' \in A ,~ \forall b,b' \in B ~.
\end{equation}
\bigskip

Let $\twa$ be the set of endomorphisms of $B\tn A$ which are morphisms in $\m$; $\twa$ is a unital algebra with the composition of morphisms in $\m$.

\begin{thm}
There is an algebra isomorphism $G:\twa\ra A\tn B^{op}$.
\end{thm}
\begin{proof} \textit{(sketch).}
Given $\ta'\in\twa$ we define $G(\ta'):=\ta(1_B\tn 1_A)$. Conversely, fixed an element $\bar{a} \ot \bar{b} \in A \ot B$, we define the inverse map $G^{-1}$ to be  
$$[G^{-1}(\bar{a} \tn \bar{b})](a\tn b) := \bar{a} a\tn b \bar{b} \, .$$ 
\end{proof}

\begin{prop}\label{prop-G}
Given a twist map $\ta'\in\tw$, the associated twisted algebra $A\tn_\ta B=(A\tn B,m_\ta)$ is unital if and only if $\ta (1_B\tn 1_A)$ is invertible in $A\tn B$, and in this case the unit map $\tu$ is defined by $\tu(1)=(\taa{1_A} \tn \taa{1_B})^{-1}$.
\end{prop}
\medskip

We call an algebra \textit{\utia}  if it is isomorphic to its opposite algebra. 
Whenever one among $A$ and $B$ is \utia, if we consider two twists whose resulting twisted tensor algebras are unital, then also the algebra corresponding to the composition of the two twists admits unit: 

\begin{lem}
\label{invstar-alg}
Let $A,B$ be algebras, and suppose at least one of them is \utia ~(say $B \simeq B^{op}$). Then an element $ a \ot b \in A \ot B$ is invertible in $A \ot B$  if and only if it is invertible as an element in  $A \ot B^{op}$.
\end{lem}

Summarizing,  a twist $\ta' \in \tw$ is invertible (and hence the associated twisted product $m_\ta$ can be `untwisted' with the inverse twist) if and only if $\ta (1_B \ot 1_A)$ is invertible in $A \ot B^{op}$. On the other hand  $(A \ot B, m_\ta)$ admits unit if and only if $\ta (1_B \ot 1_A)$ is invertible  in $A \ot B$.

A last observation is that if $\ta$ is not invertible, in particular not injective and $\bar{a} \ot \bar{b} \in ker(\ta)$, then the  twisted algebra $A \ot_\ta B$ has zero divisors:
$(a \ot \bar{b}) \cdot_\ta (\bar{a} \ot b)= 0$, for all $a \in A, ~ b \in B.$
\bigskip

We return to the results of universality addressed in Sect. \ref{sec-univ-alg} for generic twists,  now in the case of twists $\ta'\in\tw$. When the twisted algebra $A \ot_\ta b$ is unital,   we have a second pair of inclusion maps $h_A$ and $h_B$ in addition to the inclusion maps 
$i_A,~i_B$  defined in \eqref{proj-alg}.

\begin{lem}
Let $\ta'\in\tw$ such that $A \ot_\ta B$ is unital. Let $z_A \ot z_B$ be the unit element.  Then the maps $h_A: A \ra A\tn_{\ta} B$ and 
$h_B:A \ra A \tn_{\ta} B$ defined as
\begin{equation}\label{proj-h}
h_A(a) := z_A a \ot z_B \, , \qquad h_B(b) := z_A \ot b~ z_B \, , \quad a \in A, ~~b \in B  
\end{equation}
are algebra morphisms.
\end{lem}

\begin{lem}
Let $\ta'\in\tw$ such that $A \ot_\ta B$ is unital with unit element $z_A \ot z_B$. The maps $\nu:= m_\ta (i_A\tn h_B) :A \ot B \ra A\tn_{\ta} B$ and $\sigma:= m_\ta (h_A\tn i_B):A \ot B \ra A\tn_{\ta} B$ are isomorphisms of vector spaces.
\end{lem}
\begin{proof} \textit{(sketch)}. Let $a \in A$, $b \in B$. By definition of the map $\nu$, we have $\nu(a \ot b)= a z_A \ot b z_B$. It is invertible with inverse 
$\nu^{-1}(a \ot b)= a \taa{(1_A)} \ot b \taa{(1_B)}$.   
Analogously, $\sigma(a \ot b)=z_a a \ot z_B b$ and $\sigma^{-1}(a \ot b)= \taa{(1_A)} a \ot \taa{(1_B)}b$.
\end{proof}

\begin{cor}\label{cor-iso-alg}
Let $\ta'\in\tw$ such that such that $A \ot_\ta B$ is unital and the unit element is invertible in $A \ot B^{op}$. Then $\widetilde{\ta}=\mu^{-1} m_\ta (i_B\tn i_A)$ is a normal twist, different from $\tau$, and we have the double isomorphisms $A\tn_{\ta} B\simeq A\tn_{\widetilde{\ta}} B\simeq A\tn B$.
\end{cor}
\bigskip

We conclude by stressing that the twists $\ta' \in \twa$ can be composed with solutions of \eqref{Oeq} to produce new twists (cf. Sect. \ref{se-nt}):

\begin{prop}\label{prop-nt-alg}
Consider two twist maps $\chi$ and $\ta$. Suppose that $\chi$ is a solution of \eqref{Oeq} and $\ta'\in\tw$. Then the twist given by the composition $\nt:=\chi \circ \ta^\prime $ is a solution of \eqref{Oeq}.
\end{prop}
\bigskip

\subsection*{Acknowledgments}
We are pleased to thank Tomasz Brzezi\'nski for very useful remarks and correspondence. 
We also thank Gigel Militaru for comments.
LSC was partially supported by INDAM under `Borse di studio estero 2010-2011'. 
Both authors acknowledge support by the National Research Fund, Luxembourg.

\end{document}